\documentclass[11pt,oneside]{amsart}
\usepackage{fouriernc} 

\usepackage{Definitions}    
\usepackage{CLMPZ}
\usepackage{PageSetup}      
\usepackage{Environments}   
\usepackage{mathtools}
\usepackage{stmaryrd}       
\usepackage{leftidx}
\usepackage{tensor}
\usepackage{subcaption}
\usepackage{soul}
\usepackage{ifthen}

\usepackage{tikz} 
\usepackage{tikz-3dplot}
\usetikzlibrary{decorations.pathmorphing}

\newcommand{\dtref}[2]{#1}

\definecolor{coa}{HTML}{77aadd}
\definecolor{cob}{HTML}{99DDFF}
\definecolor{coc}{HTML}{ee8866}
\definecolor{cod}{HTML}{FFAABB}
\definecolor{coe}{HTML}{bbcc33}
\definecolor{cof}{HTML}{44bb99}
\definecolor{cog}{HTML}{eedd88}
\definecolor{coh}{HTML}{DDDDDD}

\ifdraft{

\usepackage{DraftSetup}

} 
{}

\hypersetup{
 pdfkeywords={latex starter,template},
 pdfauthor={Niles Johnson},
}

\usepackage{tikz-cd} 
\usetikzlibrary{decorations.pathmorphing}



\begin{document}

\title{Spectral Waldhausen categories, the $S_\bullet$-construction, and the Dennis trace}

\author[J. A. Campbell]{Jonathan A. Campbell}
\email{jonalfcam@gmail.com }
\author[J. A. Lind]{John A. Lind}
\email{jlind@csuchico.edu}
\address{Department of Mathematics and Statistics, California State University, Chico, CA USA}
\author[C. Malkiewich]{Cary Malkiewich}
\email{malkiewich@math.binghamton.edu}
\address{Department of Mathematical Sciences, Binghamton University, PO Box 6000, Binghamton, NY 13902}
\author[K. Ponto]{Kate Ponto}
\email{kate.ponto@uky.edu}
\address{Department of Mathematics, University of Kentucky, 719 Patterson Office Tower, Lexington, KY USA}
\author[I. Zakharevich]{Inna Zakharevich}
\email{zakh@math.cornell.edu}
\address{587 Malott, Ithaca, NY 14853}
\maketitle

\begin{abstract}We give an explicit point-set construction of the Dennis trace map from the $K$-theory of endomorphisms $K\End(\mathcal{C})$ to topological Hochschild homology $\mathrm{THH}(\mathcal{C})$ for any spectral Waldhausen category $\mathcal{C}$.  We describe the necessary technical foundations, most notably a well-behaved model for the spectral category of diagrams in $\mathcal{C}$ indexed by an ordinary category via the Moore end.  This is applied to define a version of Waldhausen's $S_{\bullet}$-construction for spectral Waldhausen categories, which is central to this account of the Dennis trace map.

Our goals are both convenience and transparency---we provide all details except for a proof of the additivity theorem for $\mathrm{THH}$, which is taken for granted---and the exposition is concerned not with originality of ideas, but rather aims to provide a useful resource for learning about the Dennis trace and its underlying machinery.
\end{abstract}

\setcounter{tocdepth}{1}
\tableofcontents

\section{Introduction}

This is an account of foundational material on the Dennis trace 
--- specifically, the explicit point-set model of the Dennis trace from \cite{blumberg_mandell_unpublished,dundas_goodwillie_mccarthy}. While working on a recent project \cite{clmpz-dt}, the authors found it useful to compile the details of several of these foundational results, in the particular setting of Waldhausen categories enriched in orthogonal spectra. This paper is the result of those efforts, and we are sharing it in the hope that it will be useful for others.  It is also written with the goal of serving as an entry point 
into the area of trace methods in $K$-theory, so we have included a historical overview in \S\ref{sec:history} that surveys the literature.  Since that section provides the background and motivation, we begin here with a quick summary of the paper itself.

A spectral Waldhausen category is a category enriched in orthogonal spectra, along with an associated underlying ordinary category $\cC_{0}$ that has a compatible Waldhausen structure (\cref{def:spectrally_enriched_waldhausen_category}).
In this paper we develop the necessary background to define the Dennis trace map
\begin{equation}\label{eq:intro_dennis_trace}
\trc \colon K(\uncat{\cC}) \arr \THH(\cC)
\end{equation}
for these spectral Waldhausen categories.  When $\mc{C} = \tensor[^A]{\Perf}{}$ is the spectral category of perfect modules over a ring spectrum $A$, our definition agrees with previous definitions of the Dennis trace map $K(A) \arr \THH(A)$ for $A$.

Our construction of the Dennis trace map relies on a version of Waldhausen's $S_{\bullet}$-construction for spectral Waldhausen categories.  At each simplicial level, the result $S_{\bullet}\cC$ is again a spectral Waldhausen category, as is the iterated $S_{\bullet}$-construction $S_{\bullet}^{(n)}\cC \coloneqq S_{\bullet} \dotsm S_{\bullet} \cC$, and so we can define $\THH(S_{\bullet}^{(n)}\cC)$. The inclusion of endomorphisms into the cyclic bar construction
\[
\bigvee_{\substack{f \colon c \to c \\ \text{in $\End(\uncat{\cC})$} }} \bbS \arr \THH(\cC)
\]
can then be applied to the iterated $S_{\bullet}$-construction, giving a map of multisimplicial spectra
\[
\Sigma^{\infty} \ob \End(w_{\bullet}S_{\bullet}^{(n)}\cC_{0}) \arr \THH(w_{\bullet}S_{\bullet}^{(n)}\cC).
\]
Taking the geometric realization of this map and applying the additivity theorem for $\THH$ yields the Dennis trace map
\[
K(\End(\cC_{0})) \arr \THH(\cC)
\]
out of endomorphism $K$-theory. Precomposing with the inclusion of identity endomorphisms defines the Dennis trace map on $K(\uncat{\cC})$ as in \eqref{eq:intro_dennis_trace}.

The bulk of the paper is concerned with defining $S_{\bullet}\cC$ for a spectral Waldhausen category $\cC$.  To this end, we construct a spectral category $\Fun(I, \mc{C})$ of diagrams indexed by an ordinary category $I$ with values in a spectral category $\mc{C}$ equipped with a base category $\uncat{\mc{C}}$.  The objects of $\Fun(I, \mc{C})$ are functors of ordinary categories $I \arr \uncat{\mc{C}}$. Following McClure-Smith \cite{mcclure_smith} and Blumberg-Mandell \cite{blumberg_mandell_unpublished}, the mapping spectrum between two diagrams is built out of the Moore end. The properties of the construction are summarized in \cref{thm:moore_end}; the essential points are that $\Fun(I, \mc{C})$ is functorial in both variables, the spectral category $\Fun(\ast, \mc{C})$ has the same homotopy type as $\mc{C}$, and when $I = [k] = \{0 \to \dotsm \to k\}$ is a poset category, then the mapping spectrum $\Fun(I, \mc{C})(\phi, \gamma)$ between two diagrams $\phi, \gamma \colon I \arr \uncat{\mc{C}}$ is equivalent to the homotopy limit of a zig-zag of the following form.
\[ \xymatrix @R=1em @!C=4em {
	\cC(\phi_0,\gamma_0) \ar[rd] && \ar[ld] \cC(\phi_1,\gamma_1) \ar[rd] & \cdots & \ar[ld] \cC(\phi_k,\gamma_k) \\
	& \cC(\phi_0,\gamma_1) && \cdots &
	}
\]

Diagrams in $\cC$ indexed by $I = [k] \times [k]$ are $k \times k$ commuting grids of morphisms in $\cC_{0}$.  Restricting to those diagrams that encode a sequence
\[
\ast \arr a_1 \arr \dotsm \arr a_k
\]
of cofibrations and their quotients gives the spectral category $S_{k}\cC$.

The other technical input for the Dennis trace map is the additivity theorem for $\THH$.  This is the only aspect of the construction that we do not develop from scratch, instead referring to the proofs in the literature \cite{dundas_mccarthy,dundas_goodwillie_mccarthy,blumberg_mandell_published,blumberg_mandell_unpublished}. However, see the companion paper \cite{clmpz-dt} for a succinct proof of additivity in the context of spectral Waldhausen categories, using trace methods in bicategories.

We also do not spend a significant amount of time on the lift to topological restriction homology ($\TR$) and topological cyclic homology ($\TC$). The lift to $\TR$ is treated in detail in the companion paper \cite{clmpz-dt}. The lift to $\TC$ is done in a similar way, only one works with cyclotomic spectra in the sense of \cite{madsen_survey,blumberg_mandell_cyclotomic} instead of restriction systems in the sense of \cite{clmpz-dt}. See e.g. \cite{madsen_survey,blumberg_mandell_unpublished} for more details. Note that the trace to $\TR$ is defined on all of $K(\End(\uncat{\cC}))$ but the trace to $\TC$ is only defined on $K(\uncat{\cC})$.

\subsection{Organization}

To help readers who are new to the area and describe the context of the paper, \S\ref{sec:history} discusses the history of the Dennis trace map, summarizes some of its applications in $K$-theory, and provides a guide to the literature.  This section is independent of the rest of the paper. We review background material on spectral categories in \S\ref{sec:spectral_cats}, including a careful treatment of the central example of the spectral category $\tensor[^A]{\Mod}{}$ of module spectra over a ring spectrum $A$.  The most technically demanding portion of the paper is \S\ref{ex:fun_cat}, where we define the spectral category $\Fun(I, \cC)$ of $I$-diagrams in $\cC$ and establish its relevant properties.  The $S_{\bullet}$-construction for spectral Waldhausen categories is discussed in \S\ref{sec:SWC_and_Sdot}.  In \S\ref{sec:properness}, we analyze a general method for constructing symmetric spectra from multisimplicial objects, as in the definition of Waldhausen's $K$-theory
\begin{equation}\label{eq:ktheory_intro}
	K(\uncat\cC)_{n} = \bigl| \ob w_{\bullet} S^{(n)}_{\bullet} \uncat{\cC} \bigr|,
\end{equation}
and explain how to take the left derived functor of this process, so that it is homotopy invariant.  These foundations are put to work in \S\ref{sec:dennis_trace}, where we define the Dennis trace map.  \cref{sec:model_structures} discusses model category structures on symmetric-orthogonal bispectra, and explains how to move between bispectra and other models for the stable homotopy category.

\subsection{Acknowledgments}

JC and CM would like to thank Andrew Blumberg, Mike Mandell, and Randy McCarthy for helpful conversations about this paper, and for general wisdom about trace methods.
KP was partially supported by NSF grant DMS-1810779 and the University of Kentucky Royster Research Professorship.
The authors thank Cornell University for hosting the initial meeting which led to this work.

\section{Historical overview}\label{sec:history}{}
 Quillen defined the higher algebraic $K$-groups as the homotopy groups
\[
K_n(A) = \pi_n (K_0(A) \times BGL(A)^{+})
\]
of a space built out of the plus construction of the classifying space of the infinite general linear group $GL(A) = \colim_k GL_k(A)$ \cite{quillen,quillen_plus}.  Quillen's definition sparked a revolution in the conceptual understanding of algebraic $K$-theory, but concrete calculations, beyond his computation of the $K$-theory of finite fields \cite{quillen:finite}, remain quite difficult.

The Dennis trace map has been one of
the most fruitful tools for computations of higher algebraic $K$-groups.
It was developed to approximate algebraic $K$-theory by invariants that are easier to compute.

\subsection{Hochschild homology}
The first construction of the Dennis trace occurs in unpublished work of Keith Dennis from the late 1970s. It is a homomorphism of graded abelian groups
\begin{equation}\label{eq:og_dt}
K_*(A) \arr \HH_*(A)
\end{equation}
from the algebraic $K$-theory groups to the Hochschild homology groups of a ring $A$.

The {\bf Hochschild homology groups} are defined as the homology groups of the
\emph{cyclic bar construction}, $B^{\cy}A$.  This is a simplicial abelian group
with $B^{\cy}_q \coloneqq A^{\otimes q}$, with face maps
$d_i \colon B^{\cy}_{q}A = A^{\otimes (q + 1)} \arr A^{\otimes q} = B^{\cy}_{q -
  1}A$ defined by
\begin{equation}\label{eq:hoch}
d_{i}(a_0 \otimes \dotsm \otimes a_{q}) = \begin{cases}
a_0 \otimes \dotsm \otimes a_i a_{i + 1} \otimes \dotsm \otimes a_{q} \qquad &\text{for $0 \leq i < q$} \\
a_{q}a_{0} \otimes \dotsm \otimes a_{q - 1} \qquad &\text{for $i = q$.}
\end{cases}
\end{equation}
In the case of a group ring $\bbZ[G]$, there is a canonical homomorphism from group homology to Hochschild homology
\[
H_*(BG; \bbZ) \arr \HH_*(\bbZ[G])
\]
defined on the bar construction 
by
\begin{align*}
\bbZ\{B_{q}G\} = \bbZ\{G^{q}\} &\arr \bbZ[G]^{\otimes (q + 1)} \\
(g_1, \dotsc, g_q) &\longmapsto (g_{q}^{-1}\dotsm g_{1}^{-1}) \otimes g_1 \otimes \dotsm \otimes g_q.
\end{align*}
In particular, for $G = GL_n(A)$, we can compose with the canonical ring homomorphism $\bbZ[GL_n(A)] \arr M_n(A)$ and the multitrace
\begin{align*}
\tr \colon M_{n}(A)^{\otimes (q + 1)} &\arr A^{\otimes (q + 1)} \\
g_0 \otimes \dotsm \otimes g_q &\longmapsto \sum_{i_0, \dotsc, i_{q} = 1}^{n} (g_0)_{i_0i_1} \otimes \dotsm \otimes (g_q)_{i_qi_0}
\end{align*}
to get a composite
\begin{equation}\label{eq:og_dt_formula}
H_*(BGL_n(A); \bbZ) \arr \HH_*(\bbZ[GL_n(A)]) \arr \HH_*(M_n(A)) \oarr{\mathrm{tr}} \HH_*(A).
\end{equation}
Taking the colimit as $n \rightarrow \infty$ and then precomposing with the Hurewicz map
\[
K_*(A) = \pi_*(BGL(A)^{+}) \arr H_*(BGL(A)^{+}; \bbZ) \cong  H_*(BGL(A); \bbZ)
\]
gives the original construction of the Dennis trace map \eqref{eq:og_dt} in positive degrees.

\subsection{Topological Hochschild homology}
Inspired by Goodwillie's ideas and Waldhausen's ``brave new algebra'' of ring spectra as arithmetic objects, B\"okstedt defined a lift of the Dennis trace of the form
\begin{equation}\label{eq:bok_dt}
\trc \colon K(A) \arr \THH(A).
\end{equation}
Now each of these terms is a spectrum, rather than a sequence of abelian groups, and $\THH$ is the {\bf topological Hochschild homology} of the ring or ring spectrum $A$. Essentially, $\THH$ is defined the same way as $\HH$ but with smash products $\sma_{\bbS}$ over the sphere spectrum instead of the tensor products $\otimes_{\bbZ}$. More precisely, for each ring spectrum $A$ the spectrum $\THH(A)$ is defined to be the realization of the cyclic bar construction
\[
[q] \longmapsto B^{\mathrm{cy}}_{q}(A) = A^{\sma (q + 1)}, \; \; \text{with face maps as in \eqref{eq:hoch}.}
\]
B\"okstedt's original definition is equivalent to this but defines $A^{\sma (q+1)}$ in a more elaborate way using a homotopy colimit \cite{bokstedt_thh,bokstedt_thh_Z_Z/p}.

When $A$ is an ordinary ring this refined Dennis trace can be constructed by a direct generalization of the formula \eqref{eq:og_dt_formula}. For ring spectra it is often more convenient to use the ``inclusion of endomorphisms" description from the introduction. See  \cite{blumberg_mandell_unpublished,blumberg_mandell_published} for further discussion and \cite{dundas_goodwillie_mccarthy} for a comparison of these two approaches. The latter approach has the advantage that it generalizes to spectral categories, and therefore models the Dennis trace for stable $\infty$-categories \cite{bgt}.

The trace to $\THH$ already gives more information than Dennis's original construction.  For example, when $A = \bbZ$ the Dennis trace map to $\THH$ is surjective on homotopy groups \cite{rognes_trace}. Using B\"okstedt's calculation
\[
\pi_{n}\THH(\bbZ) = \begin{cases}
\bbZ \quad & \text{if $n = 0$} \\
0 \quad & \text{if $n = 2k > 0$} \\
\bbZ/k \quad & \text{if $n = 2k - 1$,}
\end{cases}
\]
this implies that $K_n(\bbZ)$ is nontrivial for $n$ odd. On the other hand, the
ordinary Hochschild homology of the integers is concentrated in degree zero, so
cannot detect $K$-theory classes in positive degree.

\subsection{Further refinements}
At the Hochschild homology level, Connes showed that the natural cyclic permutations on the Hochschild complex come from an action of the circle group $S^1$; taking homology with respect to the circle action defines a new theory called cyclic homology, which is a non-commutative version of de Rham cohomology \cite{connes_noncommutative_IHES,loday_quillen,tsygan,HoKoRo}. {\bf Cyclic homology} $\HC(A)$ can be identified with the homotopy orbits of the $S^1$-action on the Hochschild complex, whereas the homotopy fixed points define a variant called {\bf negative cyclic homology} $\HN(A)$ \cite{hoyois_circle,kassel_cyclic,jones_cyclic}:
\[
\HC(A) \cong (B^{\cy}A)_{hS^1} \qquad \HN(A) \cong (B^{\cy}A)^{hS^1}.
\]

The Dennis trace can be lifted to a map from $K$-theory to negative cyclic homology,
\begin{equation}\label{eq:neg_dt}
K_*(A) \arr \HN_*(A).
\end{equation}
It is helpful to think of this map as a generalization of the Chern character.

At the $\THH$ level, the analogous notion is {topological cyclic homology} (TC), originally worked out by B\"okstedt-Hsiang-Madsen \cite{bokstedt_hsiang_madsen,madsen_survey}. It is not the homotopy $S^1$-fixed points of $\THH$, but instead the homotopy limit of the fixed points $\THH^{C_n}$ under certain actions by cyclic groups $C_n$. The Dennis trace to $\THH$ can be shown to factor through these fixed points, giving the {\bf cyclotomic trace}
\[
\trc \colon K(A) \arr \TC(A).
\]
By its construction, there is a forgetful map $\TC \to \THH$ along which this cyclotomic trace becomes the Dennis trace.

The reason for this divergence between the definitions of $\HN$ and of $\TC$ is purely computational. Phrasing the results in modern language, Goodwillie showed that if $A \to B$ is a map of connective ring spectra such that $\pi_0A \to \pi_0B$ has nilpotent kernel, then the Dennis trace induces a homotopy cartesian square after rationalization \cite{goodwillie_relative_ktheory}
\begin{equation}\label{eq:goodwillie_thm}
\begin{tikzcd}
K(A)_{\bbQ} \ar[d] \ar[r, "\trc"] & \HN(A_{\bbQ}) \ar[d] \\
K(B)_{\bbQ} \ar[r, "\trc"] & \HN(B_{\bbQ}).
\end{tikzcd}
\end{equation}
The integral version of this theorem is due to Dundas-Goodwillie-McCarthy \cite{dundas_relative,mccarthy_relative,dundas_goodwillie_mccarthy}: under the same hypotheses, the square induced by the cyclotomic trace
\begin{equation}\label{eq:dundas_mccarthy}
\begin{tikzcd}
K(A) \ar[d] \ar[r, "\trc"] & \TC(A) \ar[d] \\
K(B) \ar[r, "\trc"] & \TC(B)
\end{tikzcd}
\end{equation}
is homotopy cartesian. These results are important for computations since they allow us to transfer computations from one ring to all other ``nearby'' rings. See also \cite{beilinson,CMM_henselian} for more recent results of Dundas-Goodwillie-McCarthy type.

It's worth noting that while the original treatment of cyclotomic spectra and the cyclotomic trace required deep technical use of equivariant stable homotopy theory and the particular properties of B\"okstedt's model for $\THH$ (see e.g. \cite[Rmk. 2.5.9]{madsen_survey}), in recent years alternative foundations have become available.  For example, the Hill-Hopkins-Ravenel norm allows a more direct use of the cyclic bar construction in building the cyclotomic structure on $\THH$ \cite{ABGHLM_TC,dmpsw}.  Nikolaus and Scholze have also given a reformulation of a cyclotomic spectrum as an $S^1$-spectrum $X$ equipped with $S^1$-equivariant maps $X \arr X^{tC_p}$ from $X$ to the Tate construction (the homotopy cofiber of the norm map $X_{hC_p} \to X^{hC_p}$) \cite{nikolaus_scholze,hesselholt_nikolaus_TC}.  Their new definition works at the level of the underlying $\infty$-categories and provides a model-independent construction of the cyclotomic trace.  Other recent work includes an algebreo-geometric interpretation of the cyclotomic trace \cite{ayala_mazel-gee_rozenblyum_trace}.

\subsection{A few classical applications}
The cyclotomic trace map is the source of a great deal of computational knowledge about algebraic $K$-theory, in no small part because of the Dundas-Goodwillie-McCarthy theorem \eqref{eq:dundas_mccarthy}.

The original paper on the cyclotomic trace \cite{bokstedt_hsiang_madsen} describes the $\TC$ of spherical groups rings by constructing a homotopy cartesian square
\[
\begin{tikzcd}
\TC(\bbS[\Omega X])^{\wedge}_{p} \ar[r] \ar[d] & \Sigma\Sigma^{\infty}_+ ((\cL X)_{hS^1})^{\wedge}_{p} \ar[d, "\text{Trf}"] \\
\Sigma^{\infty}_{+} \cL X^{\wedge}_{p} \ar[r, "\id - \Delta_p"] & \Sigma^{\infty}_{+} \cL X^{\wedge}_{p}
\end{tikzcd}
\]
where $\Delta_p$ is the action of the $p$-th power map $z \mapsto z^p$ of the circle on the free loop space $\cL X$, and the right vertical map is the $S^1$-transfer map.  In particular, the topological cyclic homology of the sphere spectrum splits as
\[
\TC(\bbS)^{\wedge}_{p} \simeq \bbS^{\wedge}_{p} \vee \Sigma (\bbC{}P^{\infty}_{-1})^{\wedge}_{p},
\]
where the Thom spectrum $\bbC{}P^{\infty}_{-1}$ may be identified with the homotopy fiber of the transfer map $\Sigma \Sigma^{\infty}_{+} \bbC{}P^{\infty} \to \bbS$, and features prominently in the solution of Mumford's conjecture on the stable cohomology of mapping class groups \cite{madsen_weiss}.

Using Waldhausen's splitting $K(\bbS[\Omega X]) \simeq \Sigma^{\infty}_{+}X \vee \mathrm{Wh}(X)$ into the stable homotopy type of $X$ and the Whitehead spectrum of $X$, B\"okstedt-Hsiang-Madsen used the cyclotomic trace to nearly describe $\mathrm{Wh}(\ast)$ at odd regular primes $p$, a task later completed by Rognes \cite{rognes_whitehead}, and proved the $K$-theoretic analog of the Novikov conjecture: the assembly map $K(\bbZ) \sma G_+ \arr K(\bbZ[G])$ is a rational equivalence  for a large class of groups.  Along with the equivalence of relative theories provided by the homotopy cartesian square \eqref{eq:dundas_mccarthy} for the augmentation map $\bbS[\Omega X] \arr \bbS$ in the case of a simply connected space $X$ \cite{BCCGHM}, these results have concrete geometric applications in pseudoisotopy theory.  Recent work of Blumberg-Mandell \cite{blumberg_mandell_homotopy_K(S)} expands on these results to give a $p$-local splitting of the algebraic $K$-groups $K_*(\bbS)$ into the homotopy groups of known entities at all odd primes $p$.

The cyclotomic trace has also been used to compute the algebraic $K$-theory of discrete rings.  Hesselholt and Madsen showed that \cite{hesselholt_madsen}
\[
\TC(\bbF_p; p) \simeq H\bbZ_{p} \vee \Sigma^{-1} H\bbZ_{p},
\]
which agrees with Quillen's computation of $K(\bbF_p)^{\wedge}_{p}$ in non-negative degrees.  This agreement holds more generally for any perfect field $k$ of positive characteristic, and it then follows from the Dundas-Goodwillie-McCarthy theorem \eqref{eq:dundas_mccarthy} that for $A$ a finitely generated algebra over the $p$-typical Witt vectors $W(k)$, the cyclotomic trace induces an equivalence
\[
\trc \colon K(A)^{\wedge}_{p} \oarr{\sim} \tau_{\geq 0} \TC(A)^{\wedge}_{p}
\]
from $p$-complete $K$-theory to the connective cover of the $p$-completion of $\TC$.  This comparison result was used by Hesselholt-Madsen to compute the $K$-theory of complete discrete valuation fields of characteristic zero with perfect residue field $k$, thereby verifying the Quillen-Lichtenbaum conjecture in these cases \cite{hesselholt_madsen_Klocalfields}.  For example, if $F$ is a finite extension of $\bbQ_{p}$, then there is an isomorphism
\[
K_n(F; \bbZ/m) \cong K_n(k; \bbZ/m)\oplus K_{n - 1}(k; \bbZ/m)
\]
of $K$-theory groups with mod $m$ coefficients when $m$ is prime to $p$.

Trace methods have effectively computed $K$-theory for a variety of other discrete rings \cite{hesselholt_madsen_truncated,hesselholt_polytopes,hesselholt_larsen_lindenstrauss,speirs} and ring spectra, such as the connective complex $K$-theory spectrum $ku$ and its variants \cite{ausoni_rognes_Kku,ausoni_rognes_Kmorava,ausoni_rognes_rational}.

\section{Spectral categories}\label{sec:spectral_cats}

We now begin the foundational material on the Dennis trace. We discuss our conventions for spectral categories, including the important example of the spectral category of modules over a ring spectrum, and define $\THH$ of spectral categories.

\subsection{Basic definitions}
\begin{defn}
  A {\bf spectral category} $\mc C$ is a category enriched in the symmetric monoidal category of orthogonal
  spectra \cite{mandell_may_shipley_schwede}.  In more detail, this means that for every ordered pair of objects $(a, b)$ of $\mc C$, there is a spectrum $\mc C(a, b)$, which is thought of as the spectrum of maps from $a$ to $b$, as well as a unit map $\bbS \to \mc C(a, a)$ from the sphere spectrum for every object $a$, and composition maps
  \[ \mc{C}(a,b) \sma \mc{C}(b,c) \arr \mc{C}(a,c) \] that are strictly
  associative and unital.  A spectral category $\cC$ is {\bf pointwise cofibrant} if
  every mapping spectrum $\cC(a,b)$ is cofibrant in the stable model structure on orthogonal spectra \cite[\S 9]{mandell_may_shipley_schwede}.

  A {\bf functor of spectral categories} $F\colon \mc C \arr \mc D$ consists of
  a function from the objects of $\mc C$ to the objects of $\mc D$ and maps of spectra
  $F\colon \mc C(a,b) \arr \mc D(Fa,Fb)$ that respect composition and the
  unit maps. We call $F$ a {\bf Dwyer--Kan embedding} if each of these
  maps of spectra is an equivalence. As a special case, if $F$ is the identity on objects and an equivalence on each mapping spectrum we call it a {\bf pointwise equivalence}.

  Throughout, we assume that spectral categories are small, meaning that they
  have a set of objects.
\end{defn}

\begin{rmk}
  Our convention that $\mc C(a,b)$ is an orthogonal spectrum imposes no essential restriction. Any category enriched in symmetric or EKMM spectra can be turned into an orthogonal spectral category using the symmetric monoidal Quillen equivalences $(\bbP,\bbU)$ and $(\bbN,\bbN^\#)$ from \cite{mandell_may_shipley_schwede} and \cite{mandell_may}, respectively. Going back and forth only changes the spectral category by a pointwise equivalence.
\end{rmk}

\begin{example}\nopagebreak\hfill
  \begin{enumerate}
  \item Every (orthogonal) ring spectrum $A$ can be considered as a spectral category with one object.

  \item If $\uncat\cC$ is a pointed category, then there is a spectral category $\Sigma^\infty \uncat\cC$ with mapping spectra given by the suspension spectra $\Sigma^\infty \uncat\cC(a,b)$ and composition arising from $\uncat\cC$.
  \end{enumerate}
\end{example}

\begin{defn}\label{def:base_category}
  For a spectral category $\cC$, a \textbf{base category} is a pair
  $(\uncat \cC, F: \Sigma^\infty \uncat\cC \to \cC)$ where $\uncat\cC$ is an
  ordinary category and $F$ is a spectral functor that is the identity when
  restricted to object sets.  When the functor is clear from context we omit it
  from the notation.
\end{defn}

We can form such a base category $\uncat\cC$ by restricting each mapping
spectrum to level zero and forgetting the topology. There are also examples
which do not arise in this way, such as those in \S\ref{sec:spec_mod}.

Many of our techniques will require the use of
pointwise cofibrant spectral categories. Spectral categories can
always be replaced with equivalent pointwise cofibrant spectral categories using the model structure from \cite[6.1,~6.3]{schwede_shipley}.

\begin{thm} \label{thm:cofib_repl_spcat}
	There is a pointwise cofibrant replacement functor $Q$ and a pointwise fibrant replacement functor $R$ on spectral categories. In particular,
	\[Q\colon \mathbf{SpCat} \to \mathbf{SpCat}\]
	is a functor equipped with a natural transformation $q\colon Q \Rightarrow \id_{\mathbf{SpCat}}$ such that $q_{\cC}$ is a pointwise equivalence for every spectral category $\cC$.
\end{thm}

\begin{defn}\label{def:THH}
Let $\cC$ be a spectral category.
We define the {\bf
    topological Hochschild homology} or {\bf cyclic bar construction} $\THH(\mc C)$ of $\cC$ to be the geometric realization
  of the simplicial spectrum $B^{\mathrm{cy}}_\bullet(\mc C; \cX)$ given by
  \[
    B^{\mathrm{cy}}_n (\mc{C}) \coloneqq \bigvee_{c_0, \dots, c_n} \mc{C}(c_0, c_1) \sma \mc{C}(c_1, c_2) \sma \cdots \sma \mc{C}(c_{n-1}, c_n) \sma \cC(c_n, c_0).
  \]
  When $\mc C$ is pointwise cofibrant, our definition is equivalent to the previous definitions of $\THH$ in the literature,
  e.g.
  \cite{bokstedt_thh,blumberg_mandell_unpublished,dundas_goodwillie_mccarthy,nikolaus_scholze}. We can always arrange that $\cC$ is pointwise cofibrant by using the cofibrant approximation functor from \cref{thm:cofib_repl_spcat}.
\end{defn}

\subsection{The spectral category of modules over a ring spectrum $A$}\label{sec:spec_mod}

One example of a spectral category is the category $\tensor[^A]{\Mod}{}$ of
$A$-module spectra over a ring spectrum $A$, with right-derived mapping spectra between them. We proceed with a precise definition, which plays orthogonal spectra and EKMM spectra off of each other. Recall that the two models of spectra are related by a symmetric monoidal Quillen equivalence $(\bbN,\bbN^\#)$ \cite[Ch. 1]{mandell_may}.

\begin{rmk}
As a technical point, the adjunction $(\bbN,\bbN^\#)$ is with respect to the positive stable model structure on orthogonal spectra \cite[\S 14]{mandell_may_shipley_schwede}, not the usual stable model structure. However, if $X$ is cofibrant in the stable model structure then $F_1S^1 \sma X$ is cofibrant in the positive stable model structure, where $F_1S^1$ is the free orthogonal spectrum on $S^1$ at spectrum level one. So we will continue using the term ``cofibrant'' to refer to the stable model structure. Any time we need positive cofibrancy, we simply smash with $F_1S^1$.
\end{rmk}

\begin{defn}\label{modules_convention}
	Let $A$ be an orthogonal ring spectrum that is pointwise cofibrant, meaning that the underlying orthogonal spectrum is cofibrant. We define the category of {\bf $A$-modules} $\tensor[^A]{\Mod}{}$ to have as objects the cofibrant module spectra over the EKMM ring spectrum $\bbN A$. The mapping spectra $\tensor[^A]{\Mod}{}(-,-)$ are obtained by applying the lax symmetric monoidal functor $\bbN^\#$ to the mapping spectra $F_{\bbN A}(-,-)$ of $\bbN A$-modules.
\end{defn}

Note that every mapping spectrum in $\tensor[^A]{\Mod}{}$ has the expected homotopy type (i.e. is equivalent to the right-derived mapping spectrum) because EKMM spectra form a closed symmetric monoidal model category in which every object is fibrant. We get more precise control over these mapping spectra by observing that a map of orthogonal spectra
\[ X \arr \tensor[^A]{\Mod}{}(P,P') = \bbN^\# F_{\bbN A}(P,P') \]
corresponds to a map of EKMM spectra $P \sma \bbN X \to P'$ that is $\bbN A$-linear.

By a straightforward Yoneda lemma argument, $\tensor[^A]{\Mod}{}$ receives a spectral functor from the more obvious category of cofibrant orthogonal $A$-module spectra and $A$-linear maps.

\begin{lem}\label{a_into_a_modules}
	There is a functor of spectral categories to $\tensor[^A]{\Mod}{}$, from the category whose objects are cofibrant orthogonal $A$-module spectra with mapping spectra $F_A(-,-)$. The functor sends each cofibrant $A$-module $M$ to the cofibrant $\bbN A$-module $\bbN F_1S^1 \sma \bbN M$.
\end{lem}

Similarly, the category $\tensor[^A]{\Mod}{}$ receives a spectral functor from $A$ itself, by sending the single object to the $\bbN A$-module $\bbN F_1S^1 \sma \bbN A$.

\begin{lem}\label{ring_is_embedded}
	If $A$ is underlying cofibrant then this functor $A \to \tensor[^A]{\Mod}{}$ is a Dwyer-Kan embedding.
\end{lem}

\begin{proof}
	It suffices to check that one map of spectra is an equivalence: the composite
	\begin{equation}\label{eq:to_be_equiv}
  A \to \bbN^\# F_{\bbN A}(\bbN(F_1S^1 \sma A),\bbN(F_1S^1 \sma A)) \xarr\sim \bbN^\# F_{\bbN A}(\bbN(F_1S^1 \sma A),\bbN A).
  \end{equation}
	The representable functor classified by the right-hand side corresponds to $\bbN A$-linear maps
	\[ \bbN F_1 S^1 \sma \bbN A \sma \bbN (-) \to \bbN A \]
	which simplifies to $\bbS$-linear maps
	\[ \bbN F_1 S^1 \sma \bbN (-) \to \bbN A \]
	which is also represented by
	\[ F_{\bbS}(F_1S^1,\bbN^\#\bbN A) \overset\sim\leftarrow \bbN^\# \bbN A \overset\sim\leftarrow A, \]
	the homs being taken in orthogonal spectra, and using that $A \to \bbN^\#\bbN A$ is an equivalence \cite[3.5]{mandell_may}. Along these simplifications of representable functors, the original map \eqref{eq:to_be_equiv}, corresponding to the right $A$-action, goes to the class of the identity map $A \to A$. In other words, we have constructed a commutative diagram
	\[ \xymatrix{
		A \ar[d]_-{\id} \ar[r] & \bbN^\# F_{\bbN A}(\bbN(F_1S^1 \sma A),\bbN(F_1S^1 \sma A)) \ar[r]^-\sim & \bbN^\# F_{\bbN A}(\bbN(F_1S^1 \sma A),\bbN A) \ar@{<->}[d]^-\cong \\
		A \ar[r]^-\sim & \bbN^\# \bbN A \ar[r]^-\sim & F_{\bbS}(F_1S^1,\bbN^\#\bbN A)
	} \]
  whose top row is the map \eqref{eq:to_be_equiv}, which is thus an equivalence.
\end{proof}

Perfect modules also form a spectral category.
\begin{example}\label{ex:perfect_modules}
  For a ring spectrum $A$, the category of {\bf perfect $A$-modules}
  $\tensor[^A]{\Perf}{}$ is the full subcategory of $\tensor[^A]{\Mod}{}$
  spanned by the modules that are retracts in the homotopy category of finite
  cell $\bbN A$-modules. The inclusion $A \to \tensor[^A]{\Mod}{}$ factors through $\tensor[^A]{\Perf}{}$.
\end{example}

A fundamental property of $\THH$ is Morita invariance:

\begin{thm}\label{thm:THH_morita_invariant}
The canonical inclusion functor $A \arr \tensor[^A]{\Perf}{}$ induces an equivalence
\[
\THH(A) \oarr{\sim} \THH(\tensor[^A]{\Perf}{}).
\]
\end{thm}
\noindent For a proof, see \cite[5.9]{cp} or \cite[5.12]{blumberg_mandell_published}.

Recall that for ring spectra $A$ and $B$, a $(B,A)$-bimodule is a module over $B \sma A^{\op}$, or equivalently a spectrum $M$ with commuting left and right actions by $B$ and $A$, respectively. Like the proof of \cref{ring_is_embedded}, the proofs of \cref{lem:module_func,commuting_bimodules,composing_bimodules} are consequences of the Yoneda lemma through  comparisons of associated representable functors.

\begin{lem} \label{lem:module_func}
	Let $A$ and $B$ be orthogonal ring spectra.  Then for any cofibrant $(B,A)$-bimodule
	$M$, the map induced by $\bbN M \sma_{\bbN A} -$ induces a well-defined spectral functor
	\[ \tensor[^A]{\Mod}{} \to \tensor[^B]{\Mod}{}. \]
\end{lem}

Let $A$ and $B$ be orthogonal ring spectra and let $M$ a cofibrant $(B,A)$-bimodule. Then
\[ \tensor[^B]{\Mod}{}(\bbN B,\bbN F_1S^1 \sma \bbN M)\]
is also a $(B,A)$-bimodule using the composition in the spectral category $\tensor[^B]{\Mod}{}$ and the spectral functors from \cref{a_into_a_modules,lem:module_func}. Removing the technical decorations, this is just the fact that maps $B \to M$ can be pre-composed by the $B$-action on $B$ or post-composed with the $A$-action on $M$. We abbreviate this bimodule as $\tensor[^B]{\Mod}{_M}$, which is consistent with \cite[\dtref{Example 4.5}{\cref{dt-ex:bimodule_smash_functor}}]{clmpz-dt}.

\begin{lem}\label{commuting_bimodules}
	For each a cofibrant $(B,A)$-bimodule $M$, there is an equivalence of $(B,A)$-bimodules $M \to \tensor[^B]{\Mod}{_M}$, given by the map whose adjoint is the $B$-action on $M$. For each map of $(B,A)$-bimodules the resulting square
	\[ \xymatrix @R=1.5em {
		M \ar[d] \ar[r] & N \ar[d] \\
		\tensor[^B]{\Mod}{_M} \ar[r] & \tensor[^B]{\Mod}{_N}
	} \]
	strictly commutes.
\end{lem}

\begin{lem}\label{composing_bimodules}
	For each cofibrant $(C,B)$-bimodule $N$ and cofibrant $(B,A)$-bimodule $M$ the following triangle of $(C,A)$-bimodules commutes.
	\[ \xymatrix{
		N \sma_B M \ar[d] \ar[dr] & \\
		\tensor[^C]{\Mod}{_N} \sma_{\tensor[^B]{\Mod}{}} \tensor[^B]{\Mod}{_M} \ar[r] & \tensor[^C]{\Mod}{_{N \sma_B M}}
	} \]
\end{lem}

\section{The spectral category of $\cC$-valued diagrams}
\label{ex:fun_cat}

Let $\uncat\cC$ be a category and $I$ a small category. Define $\Fun(I,\uncat{\cC})$ to be the category whose objects are functors $I \to \uncat\cC$ and whose morphisms are natural transformations.

In this section, we consider a generalization of this where $\uncat\cC$ is replaced by a spectral category $\cC$, but $I$ continues to be an ordinary category. The construction we are after is present in \cite[2.4]{mcclure_smith} and \cite[2.3]{blumberg_mandell_unpublished}, but we discuss it in detail here since it is the essential technical ingredient we need in \cref{sec:s_dot} to form the $S_\bullet$ construction of a spectral Waldhausen category.

We first summarize the main properties of the construction.
\begin{thm}\label{thm:moore_end}
	If $\cC$ is a spectral category with base category $\uncat{\cC}$ (\cref{def:base_category}), and $I$ is an ordinary category, there is a spectral category $\spcat{\Fun(I,\cC)}$ with base category $\Fun(I,\uncat{\cC})$. The construction defines a functor
	\[ \Fun(-,-)\colon \mathbf{Cat}^\op \times \mathbf{SpCat} \arr \mathbf{SpCat} \]
	satisfying the following three conditions:
	\begin{enumerate}
		\item\label{thm:moore_end_we} There is a pointwise weak equivalence $\spcat{\cC} \overset\sim\arr \spcat{\Fun(*,\cC)}$ natural in $\spcat{\cC}$.
		\item\label{thm:moore_end_zz} When $I = I_0 \times \{ 0 \to \dotsm \to k \}$ then $\phi\colon I \to \uncat{\cC}$ is composed of diagrams $\phi_i\colon I_0 \to \uncat{\cC}$ for $0 \leq i \leq k$ with natural transformations between them. For two diagrams $\phi, \gamma$ these natural transformations give a zig-zag
		\[ \xymatrix @R=1em @!C=4em {
			\spcat{\Fun(I_0,\cC)}(\phi_0,\gamma_0) \ar[rd] && \ar[ld] \spcat{\Fun(I_0,\cC)}(\phi_1,\gamma_1) \ar[rd] & \cdots & \ar[ld] \spcat{\Fun(I_0,\cC)}(\phi_k,\gamma_k) \\
			& \spcat{\Fun(I_0,\cC)}(\phi_0,\gamma_1) && \cdots &
		}
		\]
		The maps
	\[\spcat{\Fun(I,\cC)}(\phi,\gamma) \to \spcat{\Fun(I_0,\cC)}(\phi_i,\gamma_i)\]
extend to an equivalence from $\spcat{\Fun(I,\cC)}(\phi,\gamma)$ to the homotopy limit of this zig-zag. We call this equivalence the \textbf{$k$th Segal map}.

 More generally, the restriction of $\spcat{\Fun(I,\cC)}(\phi,\gamma)$ to the subcategory $I_0 \times \{ i \to \cdots \to j \}$ for any $0 \leq i \leq j \leq k$ corresponds to the restriction of this zig-zag to the smaller zig-zag from $i$ to $j$.
		\item\label{thm:moore_end_iso} If $I \subset J$ is a full subcategory and every tuple of arrows passing through $J - I$ either begins or ends in $J - I$, then the spectral functor
\[\spcat{\Fun(J,\cC)} \arr \spcat{\Fun(I,\cC)}\]
 induces isomorphisms on mapping spectra between diagrams $J \to \uncat\cC$ that send $J - I$ to $* \in \ob \uncat{\cC}$.
	\end{enumerate}
\end{thm}

The theorem is proved over the remainder of this section, after which we briefly discuss the 2-functoriality of $\Fun(I, \cC)$ in the variable $\cC$ and the \emph{failure} of 2-functoriality in the variable $I$.

\subsection{Construction}

The construction of $\Fun(I,\spcat\cC)$ is by the Moore end, a variant of the homotopy end of two diagrams that has previously been considered in \cite[\S2]{mcclure_smith} and \cite[\S2.3]{blumberg_mandell_unpublished}. Like the Moore path space, it is constructed so that it forms a spectral category whose composition maps are associative, not just associative up to homotopy.

We describe the construction in stages. First recall that given two diagrams $\phi,\gamma\colon I \rightrightarrows \uncat{\cC}$, the {\bf end} is defined as the equalizer
\begin{equation}\label{eq:strict_end}
\cC(\phi,\gamma) \coloneqq \operatorname{eq}\Big( \prod_{i_0 \in \ob I} \spcat{\cC}(\phi(i_0),\gamma(i_0))
\rightrightarrows \prod_{i_0 \arr i_1}
\spcat{\cC}(\phi(i_0),\gamma(i_1))\Big).
\end{equation}
Performing this for all pairs of diagrams $I \arr \uncat{\cC}$, these define a spectral category with base category $\Fun(I,\uncat\cC)$. This defines an extension of $\Fun$ to a functor
\[\Fun: \mathbf{Cat}^\op \times \mathbf{SpCat} \arr \mathbf{SpCat}\]
and the natural map $\cC \arr \Fun(*, \cC)$ is an isomorphism of spectral categories.

The end, unfortunately, does \emph{not} suit our purposes, because we have no control over its homotopy type. We could fix this using cofibrant replacements, but that turns out to be unsatisfactory once we start changing $I$. So we pass instead to the homotopy end.

Given two diagrams in the underlying category $\phi,\gamma\colon I \rightrightarrows \uncat{\cC}$, the strict equalizer \eqref{eq:strict_end}
is a truncation of the cobar construction, a cosimplicial spectrum that at level $n$ is given by
\[ [n] \mapsto \prod_{i_0 \arr \cdots \arr i_n} \spcat{\cC}(\phi(i_0),\gamma(i_n)), \]
where the product runs over all composable $n$-tuples of arrows in $I$. The coface and codegeneracy maps are duals of the usual maps that define the categorical bar construction. The {\bf homotopy end} $h\cC(\phi,\gamma)$ is defined to be the totalization of this cobar construction.

As a simple example, if $I = \{ 0 \to 1 \}$ is a single arrow, then $\phi$ and $\gamma$ can be described as two arrows $a_0 \to a_1$ and $b_0 \to b_1$ in $\uncat\cC$. The spectrum of maps from the arrow $a_0 \to a_1$ to the arrow $b_0 \to b_1$ is then given by the strict end as the pullback
\[ \spcat\cC(a_0,b_0) \times_{\spcat\cC(a_0,b_1)} \spcat\cC(a_1,b_1), \]
while the homotopy end is the homotopy pullback. Each point in the homotopy end can be visualized as a pair of vertical arrows as below, plus a homotopy making the square commute.
\[ \xymatrix @R=1.5em @C=1.5em{ a_0 \ar@{-->}[d] \ar[r] & a_1 \ar@{-->}[d] \\ b_0 \ar[r] & b_1 } \]
Clearly it is easier to control the homotopy type of the homotopy end, but as a tradeoff it introduces a new problem: these maps cannot be composed in a strictly associative way. Stacking two of the above squares, we can compose the vertical maps and concatenate the homotopies, but for reasons of path parametrization this is not associative, only associative up to homotopy. The homotopy end therefore does not define a spectral category.

The Moore end is introduced to fix this issue in exactly the same way that the Moore loop space makes $\Omega X$ strictly associative, by letting the ``length'' of the homotopy vary. To be precise, the collapse map of $\mathbf\Delta$-diagrams $\Delta^\bullet \arr *$ induces a canonical map from the strict equalizer to the totalization
\[ \iota\colon \cC(\phi,\gamma) \arr h\cC(\phi,\gamma). \]
The {\bf Moore end} of $\phi$ and $\gamma$ is defined as the pushout
\[ \xymatrix{
	\cC(\phi,\gamma) \sma (0,\infty)_+ \ar[d] \ar[r]^{\iota \times \id} & h\cC(\phi,\gamma) \sma (0,\infty)_+ \ar[d] \\
	\cC(\phi,\gamma) \sma [0,\infty)_+ \ar[r] & h^{M}\cC(\phi,\gamma).
} \]
Note that $h^{M}\cC(\phi,\gamma)$ may be identified (levelwise) with a subset of $h\cC(\phi,\gamma) \sma [0,\infty)_+$, but it does not have the subspace topology. The deformation retract of $[0,\infty)$ onto $\{1\}$ induces a homotopy equivalence between the Moore end and the homotopy end, $h^M\spcat{\cC}(\phi,\gamma) \simeq h\spcat{\cC}(\phi,\gamma)$.

\begin{defn}
The mapping spectrum $\Fun(I, \cC)(\phi, \gamma)$ between two diagrams $\phi, \gamma \colon I \arr \uncat{\cC}$ is defined to be the Moore end $h^M\spcat{\cC}(\phi,\gamma)$.
\end{defn}

\subsection{Composition}
We now construct the composition maps for the mapping spectra in $\Fun(I, \cC)$.
For each real number $p \geq 0$, let $\Delta^n_p$ denote the $n$-simplex scaled by $p$:
\[ \Delta^n_p = \left\{ (t_0,\dotsc,t_n) \in \bbR^{n+1} : t_i \geq 0 \ \forall i, \ \sum_i t_i = p \right\}. \]
We define maps $\Delta^k_p \times \Delta^l_q \arr \Delta^{k+l}_{p+q}$ by the rule
\[ (s_0,\dotsc,s_k),(t_0,\dotsc,t_l) \mapsto (s_0,\dotsc,s_{k-1},s_k + t_0,t_1,\dotsc,t_l). \]
Fixing $p$ and $q$ but letting $k$ and $l$ vary over all pairs summing to $n$, these maps are closed embeddings that jointly cover $\Delta^n_{p+q}$ and whose overlaps are generated by face maps, allowing us to express $\Delta^n_{p+q}$ as the coequalizer
\[ \coprod_{k+l = n-1} \Delta^k_p \times \Delta^l_q \rightrightarrows \coprod_{k+l = n} \Delta^k_p \times \Delta^l_q \rightarrow \Delta^n_{p+q}. \]
This decomposition is called the \textbf{prismatic subdivision}.

Given three diagrams $\phi, \gamma, \eta\colon I \to \uncat{\cC}$, we define a composition map
\begin{equation}\label{eq:pris_comp} h\spcat{\cC}(\phi,\gamma) \sma h\spcat{\cC}(\gamma,\eta) \sma (0,\infty)^2_+ \arr h\spcat{\cC}(\phi,\eta) \sma (0,\infty)_+
\end{equation}
that extends continuously to the rest of the Moore end. The essential idea is this: for each value of $p$ and $q$, the two totalizations on the left give systems of maps
\begin{align*}
\Delta^k_p &\arr \spcat{\cC}(\phi(i_0),\gamma(i_k)) \\
\Delta^l_q &\arr \spcat{\cC}(\gamma(j_0),\eta(j_l)).
\end{align*}
For each $n$-tuple $i_0 \arr \cdots \arr i_n$ we construct a map
\[ \Delta^n_{p+q} \arr \spcat{\cC}(\phi(i_0),\eta(i_n)) \]
by splitting up $\Delta^n_{p+q}$ using the prismatic decomposition, and on each term of the form $\Delta^k_p \times \Delta^l_q$, taking the product map and applying the composition in $\spcat{\cC}$
\begin{equation}\label{eq:pris_comp_2} \Delta^k_p \times \Delta^l_q \arr \spcat{\cC}(\phi(i_0),\gamma(i_k)) \sma \spcat{\cC}(\gamma(i_k),\eta(i_n)) \arr \spcat{\cC}(\phi(i_0),\eta(i_n)).
\end{equation}

Granting that this is well-defined and extends to a continuous map (\cref{sec:moore_end_continuity}) on the Moore end, it is not hard to check that it is strictly unital and associative. For associativity we take four diagrams $\phi, \gamma, \eta, \iota\colon I \to \uncat{\cC}$, three real numbers $p,q,r > 0$, and three systems of maps
\begin{align*}
\Delta^k_p &\arr \spcat{\cC}(\phi(i_0),\gamma(i_k)) \\
\Delta^l_q &\arr \spcat{\cC}(\gamma(i_0),\eta(i_l)) \\
\Delta^m_r &\arr \spcat{\cC}(\eta(i_0),\iota(i_m)).
\end{align*}
For each $n$-tuple $i_0 \arr \cdots \arr i_n$, we check that the two maps we can construct
\[ \Delta^n_{p+q+r} \arr \spcat{\cC}(\phi(i_0),\iota(i_n)) \]
agree, using the fact that the two ways of composing prismatic subdivisions lead to the same three-fold subdivision formula
\[ (s_0,\dotsc,s_k),(t_0,\dotsc,t_l),(u_0,\dotsc,u_m) \mapsto (s_0,\dotsc,s_k + t_0,t_1,\dotsc,t_l+u_0,\dotsc,u_m). \]

The proof that this construction is appropriately functorial in $\spcat{\cC}$ and $I$ is long but entirely straightforward, so we omit the details.

\subsection{Continuity of the composition maps}\label{sec:moore_end_continuity}
This is the most demanding part of the proof of \cref{thm:moore_end}. The construction in \eqref{eq:pris_comp_2} defines a continuous map into the image of the following injective map.
\begin{equation}\label{big_nasty_injective}
\begin{aligned}
	(0,\infty)^2_+ \sma \displaystyle\prod_{[n]} F&\left( \Delta^n_+, \displaystyle\prod_{i_0 \arr \cdots \arr i_n} \spcat{\cC}(\phi(i_0),\eta(i_n)) \right) \to
\\
	&(0,\infty)^2_+ \sma \displaystyle\prod_{[n]}\displaystyle\prod_{k+l=n} F\left( (\Delta^k \times \Delta^l)_+, \displaystyle\prod_{i_0 \arr \cdots \arr i_n} \spcat{\cC}(\phi(i_0),\eta(i_n)) \right).
\end{aligned}
\end{equation}
To see that the maps in \eqref{eq:pris_comp_2}  assemble to a continuous map on the source of  \eqref{big_nasty_injective}
it is enough to show that \eqref{big_nasty_injective} is an inclusion (i.e. a homeomorphism onto its image).
It is then straightforward to check that as $n$ varies, our lifts to the source of \eqref{big_nasty_injective} glue together to give a well-defined map of totalizations.

We prove \eqref{big_nasty_injective} is an inclusion by showing it is an equalizer of two maps that restrict the $\Delta^k \times \Delta^l$ to certain faces. We do this by taking the prismatic subdivision with $p = q = 1$, whose coequalizer we call $C$, and taking maps out to get an equalizer system that is not quite the same as \eqref{big_nasty_injective}. (We use \cite[5.3]{strickland_cgwh} to argue that smashing with $(0,\infty)^2_+$ preserves the equalizer.) But this second equalizer system is homeomorphic to \eqref{big_nasty_injective}, using the evident homeomorphism $\pi_{p,q}\colon C \xarr\cong \Delta^n_{p+q}$ that varies continuously in $p$ and $q$.

We must also prove that the multiplication is continuous when one factor is in the strict end times $[0,\infty)$ and the other is in the homotopy end times $(0,\infty)$. When $p = 0$ and $q > 0$, the maps of the prismatic subdivision still make sense, but they do not form a coequalizer system of the form
\begin{equation}\label{not_coequalizer}
	\coprod_{k+l = n-1} \Delta^k \times \Delta^l \rightrightarrows \coprod_{k+l = n} \Delta^k \times \Delta^l \rightarrow \Delta^n_{p+q}.
\end{equation}
In other words, letting $C$ denote the coequalizer of \eqref{not_coequalizer} and $\pi_{p,q}$ the induced map
\[ \pi_{p,q} \colon C \to \Delta^n_{p+q}, \]
then $\pi_{p,q}$ is a quotient map but not a homeomorphism. As a result, the following analog of \eqref{big_nasty_injective} (products suppressed) is not an equalizer of the two maps that restrict $\Delta^k \times \Delta^l$ to the relevant faces.
\begin{equation}\label{bigger_nasty_injective}
\xymatrix{
	[0,\infty)_+ \sma (0,\infty)_+ \sma F\left( \Delta^n_+, A \right) \ar[r] & [0,\infty)_+ \sma (0,\infty)_+ \sma \displaystyle\prod_{k+l = n} F\left( (\Delta^k \times \Delta^l)_+, A \right),
}
\end{equation}
However, our composition rule still lifts to the source of \eqref{bigger_nasty_injective} as a map of sets. So it suffices to prove that \eqref{bigger_nasty_injective} is an inclusion.  This will follow if we can show that the map
\[ \xymatrix{ [0,\infty)_+ \sma (0,\infty)_+ \sma F\left( \Delta^n_+, A \right)
    \arr [0,\infty)_+ \sma (0,\infty)_+ \sma F\left( C_+, A \right) } \] that
sends $(p,q,f\colon \Delta^n \arr A)$ to
$(p,q,f \circ \pi_{p,q}\colon C \arr A)$ is an inclusion. It suffices to prove
this with Cartesian products inside the category of topological spaces because
$k(-)$ preserves closed inclusions by a quick diagram-chase.

Consider the map of Cartesian products
\begin{equation}\label{smaller_nasty_injective}
	[0,\infty) \times (0,\infty) \times F\left( \Delta^n_+, A \right) \arr [0,\infty) \times (0,\infty) \times F\left( C_+, A \right)
\end{equation}
with formula $(p,q,f\colon \Delta^n \arr A) \mapsto (p,q,f \circ \pi_{p,q}\colon C \arr A)$. It suffices to show that the image of the basis element $U \times V \times \{K_i \arr W_i\}$ is relatively open in the image. Take $p_0 \in U$, $q_0 \in V$, $f\colon K_i \arr W_i$. It suffices to find a basis element in the target containing the point $(p_0,q_0,f \circ \pi_{p_0,q_0})$, whose intersection with the image is completely contained in the image of $U \times V \times \{K_i \arr W_i\}$.

To do this, note that $f(K_i) \subseteq W_i$ implies that $f \circ \pi_{p_0,q_0} (\pi_{p_0,q_0}^{-1}(K_i)) \subseteq W_i$. In fact, we can find a small open interval of values of $p$ and $q$ about $p_0$ and $q_0$ such that for $p$ and $q$ in this interval,
\[ f \circ \pi_{p_0,q_0} (\pi_{p,q}^{-1}(K_i)) \subseteq W_i. \]
This is because the set of maps $g\colon C \arr \Delta^n$ such that $g^{-1}(K_i) \subseteq (f \circ \pi_{p_0,q_0})^{-1}(W_i)$ is equal to the set of maps such that $g$ sends the compact set $C - (f \circ \pi_{p_0,q_0})^{-1}(W_i)$ to the open set $\Delta^n - K_i$, which is open in the compact-open topology.

Shrinking these intervals slightly to closed intervals $I$ and $J$, the map
$I \times J \times C \to \Delta^n$ given by $(p,q,x) \mapsto \pi_{p,q}(x)$ is
continuous, therefore the inverse image of $K_i$ is closed---and thus
compact. Its projection to $C$ is then also compact, as is the union of
$\pi_{p,q}^{-1}(K_i)$ over $p \in I$ and $q \in J$ is compact. Call this union
$L_i$. The above implies that
\[ f \circ \pi_{p_0,q_0} (L_i) \subseteq W_i. \]

Shrinking one last time to open intervals $I'$ and $J'$ about $p_0$ and $q_0$, respectively, consider the open set $I' \times J' \times \{L_i \arr W_i\}$ in the target of \eqref{smaller_nasty_injective}. Every triple in this set of the form $(p,q,f' \circ \pi_{p,q})$ has the property that
\[ f' \circ \pi_{p,q}(L_i) \subseteq W_i. \]
Since $\pi_{p,q}^{-1}(K_i) \subseteq L_i$, we deduce that
\[ f'(K_i) \subseteq W_i. \]
So the intersection of $I' \times J' \times \{L_i \arr W_i\}$ with the image of \eqref{smaller_nasty_injective} lies in the image of $U \times V \times \{K_i \arr W_i\}$. Furthermore this basic open set contains the original point $(p_0,q_0,f \circ \pi_{p_0,q_0})$ by construction. Repeating for each $i$ and taking the intersection, we get a basis element in the target of \eqref{smaller_nasty_injective} whose intersection with the image is contained in the source basis element, and contains the prescribed point. This proves that \eqref{smaller_nasty_injective}, and therefore \eqref{bigger_nasty_injective}, is an inclusion.

By a slight variant of this argument (extending to an infinite product by only
ever imposing conditions on finitely many of the factors), we can also take a
product over $n$ before smashing with $[0,\infty)_+ \sma (0,\infty)_+$, giving
an inclusion
\begin{equation}\label{eq:extended_product_prismatic_equalizer}
\xymatrix{
	[0,\infty)_+ \sma (0,\infty)_+ \sma \displaystyle\prod_n F\left( \Delta^n_+, A_n \right) \ar[r] &
	[0,\infty)_+ \sma (0,\infty)_+ \sma \displaystyle\prod_n \displaystyle\prod_{k+l = n} F\left( (\Delta^k \times \Delta^l)_+, A_n \right).
}
\end{equation}
This concludes the most technical part of the proof.

\subsection{Proof that the three statements in \cref{thm:moore_end} hold}

The totalization defining the homotopy end will be Reedy fibrant, and therefore preserve weak equivalences, when the terms of the product (the mapping spectra of $\spcat\cC$) are all fibrant. Alternatively, it also has the correct homotopy type if the category $I$ is finite, in the sense that its classifying space is a finite CW complex. So if $I$ is finite, we proceed; otherwise, we first assume we have taken a pointwise fibrant replacement of $\cC$ using \cref{thm:cofib_repl_spcat} before building the Moore end.
\begin{enumerate}
	\item When $I = *$, the strict equalizer is just $\spcat{\cC}$, and its inclusion into the homotopy equalizer is an equivalence because the cosimplicial object is constant.
	\item It suffices to consider the homotopy end. When $I = I_0 \times [k]$, this is the totalization of the diagonal of the bicosimplicial space
          \[ [m,n] \mapsto \prod_{\stackrel{i_0 \arr \cdots \arr i_m \in
                I_0}{j_0 \arr \cdots \arr j_n \in [k]}}
            \spcat{\cC}(\phi(i_0,j_0),\gamma(i_m,j_n)). \] It is standard that
          this is homeomorphic to the double totalization, which is the homotopy
          end along $[k]$ of the homotopy end along $I_0$. Therefore without
          loss of generality $I_0 = *$. The zig-zag description is then fairly
          standard, see e.g. \cite[\S2.4]{blumberg_mandell_unpublished}. One way
          to prove it is to expand to a larger diagram with equivalent homotopy
          limit; for example, when $k = 2$ the larger diagram is as follows:
	\[ \xymatrix @R=1em @C=1em {
		\spcat{\cC}(\phi(0),\gamma(0)) \ar[rd] && \ar[ld] \spcat{\cC}(\phi(1),\gamma(1)) \ar[rd] && \ar[ld] \spcat{\cC}(\phi(2),\gamma(2)) \\
		& \spcat{\cC}(\phi(0),\gamma(1)) \ar[rd] && \ar[ld] \spcat{\cC}(\phi(1),\gamma(2)) & \\
		&& \spcat{\cC}(\phi(0),\gamma(2))
	}
	\]
	We model this homotopy limit as maps out of a cofibrant replacement of the trivial diagram $*$ sending each $(i,j)$ to $\Delta^{j-i}$, and face maps between them. It is then straightforward to give a homeomorphism between this model of the homotopy limit and the desired homotopy end.

	\item The given condition implies that on two diagrams $\phi$ and $\gamma$ sending every object of $J - I$ to $*$, the map of cosimplicial objects induced by $I \to J$ is a product of identity maps and zero maps, hence is an isomorphism. It therefore induces isomorphisms on the homotopy end and the strict end, and hence an isomorphism on the Moore end as well.
\end{enumerate}

\subsection{2-functoriality}\label{2-functoriality}

To conclude this section we record the behavior of the Moore end with respect to natural transformations. Given two morphisms $F,G\colon (\spcat\cC,\uncat\cC) \rightrightarrows (\spcat\cD,\uncat\cD)$ of spectral categories equipped with base categories, we define a {\bf natural transformation} $\eta\colon F \Rightarrow G$ to be a natural transformation of the underlying functors $\uncat\cC \to \uncat\cD$, subject to the condition that the following square commutes for all objects $a$ and $b$ in $\cC$
\[ \xymatrix{
	\cC(a,b) \ar[d]_-G \ar[r]^-F & \cD(Fa,Fb) \ar[d]^-{\eta(b) \circ -} \\
	\cD(Ga,Gb) \ar[r]_-{- \circ \eta(a)} & \cD(Fa,Gb).
} \]
These make the category $\mathbf{SpCat}$ of pairs $(\spcat\cC,\uncat\cC)$ into a 2-category.

The category of diagrams construction $\Fun(I,\uncat\cC)$ defines a 2-functor $\mathbf{Cat}^\op \times \mathbf{Cat} \arr \mathbf{Cat}$, in other words it respects ordinary natural transformations in both variables. Therefore the Moore end will define a 2-functor if and only if these natural transformations also respect the spectral enrichments. However, this is only true in the $\cC$ variable, and not in the $I$-variable:
\begin{prop}\label{thm:moore_end_2_functorial}
	The Moore end of \cref{thm:moore_end} defines a 2-functor
	\[ \Fun(-,-)\colon \mathbf{Cat}^\op \times \mathbf{SpCat} \arr \mathbf{SpCat} \]
	if $\mathbf{Cat}$ is given the trivial 2-category structure with only identity natural transformations.
\end{prop}
The proof is a straightforward verification and will be omitted.

Although the Moore end fails to preserve natural transformations in the first variable, it does preserve them up to homotopy. To see this, let $\alpha,\beta\colon I \rightrightarrows J$ be a pair of functors of small categories, and let $h\colon \alpha \Rightarrow \beta$ be a natural transformation. Composing with $h$ defines maps
\begin{equation}\label{underlying_module_map}
\xymatrix{ \Fun(I,\cC)(\phi,\gamma\circ\alpha) \ar[r]^-{h \circ -} & \Fun(I,\cC)(\phi,\gamma\circ\beta) }
\end{equation}
for each $\phi\colon I \to \uncat\cC$ and $\gamma\colon J \to \uncat\cC$. The failure of the enriched naturality condition then becomes the following statement: these maps commute with the left action of $\Fun(I,\cC)(-,\phi)$ but fail to commute with the right action of $\Fun(J,\cC)(\gamma,-)$. 

However, we can thicken these spectra so that the map does commute with the right action. Let $[1]$ denote the poset category $\{0 < 1\}$, let $H\colon [1] \times I \to J$ denote the functor encoding $\alpha$, $\beta$, and $h$, and let $c\colon [1] \times I \to I$ denote the projection to $I$. Then the {\bf canonical replacement} of the map \eqref{underlying_module_map} is the zig-zag
	\[ \xymatrix{
		\Fun(I,\cC)(\phi,\gamma\circ\alpha) & \ar[l]_-\sim \Fun([1] \times I,\cC)(\phi\circ c,\gamma\circ H) \ar[r] & \Fun(I,\cC)(\phi,\gamma\circ\beta)
	} \]
	where the maps restrict along the inclusions of $\{0\} \times I$ and $\{1\} \times I$ into
  $[1] \times I$.

\begin{lem}
	For each $\phi\colon I \to \uncat\cC$, the two maps of the canonical replacement each commute with the right $\Fun(J,\cC)$-action. Furthermore the diagram of spectra
	\[ \xymatrix{
		\Fun([1] \times I,\cC)(\phi\circ c,\gamma\circ H) \ar[d]_-\sim \ar[rd] & \\
		\Fun(I,\cC)(\phi,\gamma\circ\alpha) \ar@{-->}[r]^-{h \circ -} & \Fun(I,\cC)(\phi,\gamma\circ\beta)
	} \]
	commutes up to a canonical homotopy.
\end{lem}

\begin{proof}
	The maps in the canonical replacement arise from morphisms of spectral categories, hence commute with the right $\Fun(J,\cC)$-actions. By \cref{thm:moore_end}(\ref{thm:moore_end_zz}) the first Segal map gives an equivalence from $\Fun([1] \times I,\cC)(\phi\circ c,\gamma\circ H)$ into the homotopy limit of the zig-zag
	\[ \xymatrix{ \Fun(I,\cC)(\phi,\gamma\circ\alpha) \ar@{-->}[r]^-{h \circ -}  & \Fun(I,\cC)(\phi,\gamma\circ\beta) & \ar[l]_-{=} \Fun(I,\cC)(\phi,\gamma\circ\beta). } \]
	Along this equivalence, the map to $\Fun(I,\cC)(\phi,\gamma\circ\alpha)$ is the restriction to the first term and the map to $\Fun(I,\cC)(\phi,\gamma\circ\beta)$ is the restriction to the last term. The homotopy limit then provides the canonical homotopy.
\end{proof}

\section{Spectral Waldhausen categories and the $S_\bullet$ construction}\label{sec:SWC_and_Sdot}
\subsection{Spectral Waldhausen categories}\label{ssec:sp_wald}
We begin by defining the notion of a Waldhausen category with a compatible spectral enrichment---these are the input data for our version of the Dennis trace constructed in \S\ref{sec:dennis_trace}. Recall that a \textbf{Waldhausen category} is a category $\uncat{\cC}$ equipped with subcategories of cofibrations and weak equivalences such that
\begin{enumerate}
	\item every isomorphism is both a cofibration and a weak equivalence,
	\item there is a zero object $*$ and every object is cofibrant,
	\item $\uncat{\cC}$ has all pushouts along cofibrations (homotopy pushouts),
	\item the pushout of a cofibration is a cofibration, and
	\item  a weak equivalence of homotopy pushout diagrams induces a weak equivalence of pushouts.
\end{enumerate}
An \textbf{exact functor} $\uncat{\cC} \to \uncat{\cD}$ is a functor preserving the zero object, cofibrations, weak equivalences, and pushouts along cofibrations.

\begin{defn}\label{def:spectrally_enriched_waldhausen_category}
  A \textbf{spectral Waldhausen category} is a spectral category $\cC$
  with a base category $\uncat\cC$ that is equipped with a Waldhausen category structure. This data is subject to the following three conditions:
\begin{enumerate}
	\item\label{def:spectrally_enriched_waldhausen_category_zero} The zero object of $\uncat{\cC}$ is also a zero object for $\spcat{\cC}$.
	\item\label{def:spectrally_enriched_waldhausen_category_we} Every weak equivalence $c \arr c'$ in $\uncat{\cC}$ induces stable equivalences
	\[ \spcat{\cC}(c',d) \overset\sim\arr \spcat{\cC}(c,d), \qquad \spcat{\cC}(d,c) \overset\sim\arr \spcat{\cC}(d,c'). \]
	\item\label{def:spectrally_enriched_waldhausen_category_pushout} For every pushout square in $\uncat{\cC}$ along a cofibration
	\[ \xymatrix @R=1.5em{
		a \, \ar[d] \ar@{^(->}[r] & b \ar[d] \\
		c \,\ar@{^(->}[r] & d
	} \]
	and object $e$, the resulting two squares of spectra
	\[ \xymatrix @R=1.5em{
		\spcat{\cC}(a,e) \ar@{<-}[d] \ar@{<-}[r] & \spcat{\cC}(b,e) \ar@{<-}[d] \\
		\spcat{\cC}(c,e) \ar@{<-}[r] & \spcat{\cC}(d,e)
	}
	\qquad
	\xymatrix @R=1.5em{
		\spcat{\cC}(e,a) \ar[d] \ar[r] & \spcat{\cC}(e,b) \ar[d] \\
		\spcat{\cC}(e,c) \ar[r] & \spcat{\cC}(e,d)
	} \]
	are homotopy pushout squares.
      \end{enumerate}

      Let $\mathbf{SpWaldCat}$ be the category whose objects are spectral Waldhausen categories $(\spcat\cC,\uncat\cC)$. When it is clear from context, we omit $\uncat\cC$ from the notation. A morphism $(\spcat{\cC},\uncat{\cC}) \arr (\spcat{\cD},\uncat{\cD})$ consists of an exact functor
      $\uncat{F}: \uncat{C} \arr \uncat{D}$ and a spectral functor
      $\spcat{F}: \spcat{\cC} \arr \spcat{\cD}$ such that the diagram
      \[\xymatrix{
          \Sigma^\infty \uncat\cC \ar[r]^{\Sigma^\infty \uncat{F}} \ar[d] & \Sigma^\infty \uncat\cD
          \ar[d] \\
          \spcat{\cC} \ar[r]^{\spcat{F}}& \spcat{\cD}
        }\]
commutes.

    \end{defn}

\begin{example} \label{ex:Perf_as_spectral_wald_cat}
	The spectral categories $\tensor[^A]{\Perf}{}$ of perfect $A$-modules and
        $\tensor[^A]{\Mod}{}$ of all $A$-modules are both spectral Waldhausen
        categories. If desired, we can restrict along $\bbN$ and take the underlying category to be the category of cofibrant orthogonal $A$-module spectra (respectively, those that are perfect). The proof of the above three conditions crucially uses our convention in \cref{modules_convention} that the mapping
        spectra in $\tensor[^A]{\Mod}{}$ are derived using EKMM spectra. Generalizing
        this, in any spectrally enriched model category with all objects fibrant, the
        subcategory of cofibrant objects is a spectral Waldhausen category.
\end{example}

\begin{example} \label{ex:BM} If $\uncat{\cC}$ is a simplicially enriched
  Waldhausen category in the sense of \cite{blumberg_mandell_unpublished} then
  the spectral enrichment $\cC^\Gamma$ from
  \cite[2.2.1]{blumberg_mandell_unpublished} is compatible with the Waldhausen
  structure in our sense. The same is true for the non-connective enrichment
  $\cC^{\mc S}$ from \cite[2.2.5]{blumberg_mandell_unpublished} if $\uncat{\cC}$
  is enhanced simplicially enriched.
\end{example}

Recall the spectral category $\Fun(I, \cC)$ of $I$-diagrams in $\cC$ from \S\ref{ex:fun_cat}. If $\cC$ is a spectral Waldhausen category, then we can give $\Fun(I, \uncat\cC)$ the levelwise Waldhausen structure.
In other words, given functors $\phi,\gamma: I \to \uncat\cC$, a natural
transformation $\alpha \colon \phi \Rightarrow \gamma$ is a cofibration (resp. weak equivalence) if for every
$i\in \ob I$, the morphism $\alpha(i) \colon \phi(i) \to \gamma(i)$ is a cofibration (resp. weak equivalence) in $\uncat\cC$.

\begin{prop}\label{prop:fun_cat_wald}
  The functor $\Fun(-,-)$ of \cref{thm:moore_end} respects Waldhausen structures in the second entry, defining a functor
  \begin{align*}
    \mathbf{Cat}^\op \times \mathbf{SpWaldCat} &\arr \mathbf{SpWaldCat} \\
    (I, (\cC, \uncat\cC)) &\longmapsto (\Fun(I, \cC), \Fun(I, \uncat\cC)),
  \end{align*}
  where $\Fun(I, \uncat\cC)$ has the levelwise Waldhausen structure.
\end{prop}

\begin{proof}
	The axioms in \cref{def:spectrally_enriched_waldhausen_category} are invariant under pointwise equivalence, so whenever necessary we may fibrantly replace the mapping spectra of $\spcat\cC$ without changing the fact that it is a spectral Waldhausen category. Then we build the Moore end and verify the three axioms of a spectrally enriched Waldhausen category.
	\begin{enumerate}
		\item When $\phi = *$ is the zero diagram, every term in the cosimplicial spectrum is a product of zero spectra $\spcat{\cC}(*,\gamma(i_n))$, so the totalization and therefore the Moore end is also isomorphic to zero. The same argument works when $\gamma = *$.
		\item An equivalence of diagrams $\phi \to \phi'$ induces a product of equivalences
		\[\spcat{\cC}(\phi'(i_0),\gamma(i_n)) \arr \spcat{\cC}(\phi(i_0),\gamma(i_n))\] at each cosimplicial level, giving an equivalence on totalizations, and the same argument works in the variable $\gamma$.
		\item Similarly, when we put a pushout square along a cofibration in one of the variables, the resulting square of cosimplicial spectra is a homotopy pullback at each cosimplicial level, hence a homotopy pullback on the totalizations.
	\end{enumerate}
\end{proof}

The following is immediate from the definition, but will be important later for
defining the $S_\bullet$-construction:
\begin{lem} \label{lem:restrict_cofib}
	Let $(\spcat\cC,\uncat\cC)$ be a spectral Waldhausen category, and let
	$\uncat\cC'$ be a Waldhausen category with the same objects, morphisms, and
	 weak equivalences as $\uncat\cC$, and whose cofibrations are a subcategory
	of the cofibrations of $\uncat\cC$.  Then $(\spcat\cC, \uncat\cC')$ is also a
	spectral Waldhausen category.
\end{lem}

\subsection{The $S_\bullet$ construction}\label{sec:s_dot}
In this section we recall the $S_\bullet$ construction for Waldhausen categories and describe how to extend it to spectral Waldhausen categories.

\begin{defn}\label{ex:spec_w_bullet}
  Let $[k]$ denote the poset $\{0 \to 1 \to \dotsm \to k\}$, considered as a
  category.
  Let $(\spcat \cC, \uncat\cC)$ be a spectral Waldhausen category.  Write
  $w_{k} \uncat{\cC}$ for the full subcategory of $\Fun([k], \uncat{\cC})$ spanned by
  the functors that take each morphism in $[k]$ to a weak equivalence in
  $\uncat\cC$, and write $w_k\spcat\cC$ for the full subcategory of the spectral category
  $\Fun([k],\spcat\cC)$ spanned by the objects of $w_k\uncat\cC$.  Then
  $(w_k\spcat{\cC}, w_k\uncat{\cC})$ is a spectral Waldhausen category.  By the functoriality of the Moore end in the $I$ coordinate (\cref{thm:moore_end}), as $k$ varies this
  forms a simplicial object $w_{\bullet}$ in $\mathbf{SpWaldCat}$.
\end{defn}

\begin{lem}\label{w_bullet_invariance}
  The iterated degeneracy map $w_0\spcat{\cC} \to w_k\spcat{\cC}$ is a
  Dwyer--Kan equivalence of spectral categories.  In particular,
  the spectral categories $w_k\spcat{\cC}$ are all canonically Dwyer--Kan
  equivalent to $\spcat{\cC}$.
\end{lem}

\begin{proof}
  It follows from \cref{thm:moore_end}.\ref{thm:moore_end_zz} that each degeneracy functor is a Dwyer--Kan embedding. Every string of weak equivalences is equivalent to a string of identity maps, and by
  \cref{def:spectrally_enriched_waldhausen_category}.\ref{def:spectrally_enriched_waldhausen_category_we}
  this is also an isomorphism in the homotopy category coming from the spectral
  enrichment. Therefore each degeneracy functor is also a Dwyer--Kan equivalence.  The last
  statement of the lemma follows because $w_0 \spcat{\cC}$ is pointwise
  equivalent to $\spcat{\cC}$ by \cref{thm:moore_end}.\ref{thm:moore_end_we}.
\end{proof}

\begin{defn}\label{ex:spec_s_dot}  We write $(i \leq j)$, $(i = j)$, or $(i \geq j)$ for an object $(i, j)$ of $[k] \times [k]$ according to which of the given relations holds.
	If $\uncat{\cC}$ is a Waldhausen category, let $S_k \uncat{\cC}$ denote the full subcategory of $\Fun([k] \times [k], \uncat{\cC})$ whose objects are the functors which
	\begin{itemize}
		\item take each object of the form $(i \geq j)$ to the zero object $* \in \ob \uncat\cC$,
		\item take each arrow $(i \leq j) \arr (i \leq j+1)$ to a cofibration, and
		\item take each square of the form
		\[ \xymatrix @R=1.5em{
			(i \leq j) \ar[d] \ar[r] & (i \leq \ell) \ar[d] \\
			(i  = i) \ar[r] & (j \leq \ell)
		} \]
		to a pushout square (along a cofibration).
	\end{itemize}
	By the first condition, the structure of an object of $S_{k}\uncat\cC$ is encoded by a functor on the category of arrows $\Arr[k]$. The other two conditions tell us that the functor gives a sequences of cofibrations
	\[{*} \arr a_1 \arr a_2 \arr \dotsm \arr a_k \]
	in $\uncat\cC$, along with the grid formed by their quotients, as in the usual $S_\bullet$ construction \cite[\S1.3]{1126}.

	If $(\spcat{\cC}, \uncat{\cC})$ is a spectral Waldhausen category, let
        $S_k\spcat{\cC}$ denote the full subcategory of
        $\Fun([k] \times [k], \cC)$ spanned by the objects of $S_k \uncat{\cC}$. This is a spectral
        Waldhausen category, and by \cref{lem:restrict_cofib} continues to be so
        when we restrict the class of cofibrations on the base category
        $S_k \uncat{\cC}$ to match the usual ones (described by the hypotheses of Lemma 1.1.3 in \cite{1126}).  As $k$
        varies, the pairs $(S_k\spcat{\cC}, S_k\uncat{\cC})$ define a simplicial object $S_{\bullet}\spcat\cC$ in
        $\mathbf{SpWaldCat}$.
\end{defn}

\begin{rmk}
  By \cref{thm:moore_end}.\ref{thm:moore_end_iso}, there is an isomorphism of spectral categories
  \[S_1 \spcat{\cC} \cong \spcat{\Fun(*,\cC)}.\] More
  generally, we get the same spectral enrichment on $S_k\spcat{\cC}$ if we trim
  down the category $[k] \times [k]$ by removing the objects 0 and $k$ from the
  first and second copies of $[k]$, respectively. However the spectral
  enrichment using $\Arr[k]$ in place of $[k] \times [k]$ is not isomorphic, or even equivalent, to ours. The extra zero objects provide the mapping
  spectra with important nullhomotopies.
\end{rmk}

\begin{defn}\label{ex:spec_iterated_s_dot} The spectral category
  $w_{k_0}S^{(n)}_{k_1, \dotsc, k_n} \spcat{\cC} \coloneqq w_{k_0}S_{k_1}\cdots
  S_{k_n}\spcat{\cC}$ is defined by taking the full subcategory of the spectral
  category
  \[ \spcat{\Fun([k_0] \times [k_1]^2 \times \dotsm \times [k_n]^2, \cC)} \]
  spanned by those functors
  \[ [k_0] \times [k_1]^2 \times \dotsm \times [k_n]^2 \arr \uncat\cC \]
  satisfying the condition determined by iterating the $S_\bullet$ construction. As
  in \cref{ex:spec_s_dot}, we then restrict the cofibrations on the base
  category to match those that occur in the iterated $S_\bullet$-construction;
  as before, by \cref{lem:restrict_cofib} the result is a well-defined spectral
  Waldhausen category.
  By the functoriality of the Moore end in the $I$ coordinate (\cref{thm:moore_end}), as the indices $k_i$ vary the construction defines an
  $(n+1)$-fold multisimplicial object $w_{\bullet} S^{(n)}_{\bullet} \spcat{\cC}$ in $\mathbf{SpWaldCat}$.
\end{defn}

\section{Symmetric spectra built from multisimplicial sequences}\label{sec:properness}

Recall that for an ordinary Waldhausen category $\uncat\cC$, the algebraic $K$-theory $K(\uncat\cC)$ is the symmetric spectrum that at level $n$ is the geometric realization of the multisimplicial set
	\begin{equation}\label{eq:k-theory_def}
	K(\uncat\cC)_{n} = \bigl| \ob w_{\bullet} S^{(n)}_{\bullet} \uncat{\cC} \bigr|.
\end{equation}
The $\Sigma_n$-actions permute the  $S_{\bullet}$ terms and the spectrum structure maps come from the isomorphisms
\[ w_{k_0}S^{(n)}_{k_1,\dotsc,k_n} \uncat{\cC} \cong w_{k_0}S^{(n+1)}_{k_1,\dotsc,k_n,1} \uncat{\cC}. \]
In this section we formalize this method of creating a symmetric spectrum from a sequence of multisimplicial objects. As a result, we can apply the same process to $\THH(w_\bullet S^{(n)}_\bullet \cC)$ for a spectral Waldhausen category $\cC$. This is an essential maneuver in the construction of the Dennis trace in \S\ref{sec:dennis_trace}.

We also prove a properness theorem (\cref{properness-thm}), stating that the formation of symmetric spectra as in \eqref{eq:k-theory_def} can be left-derived so that it is always homotopy invariant. In other words, for such an object we can always make all of the multisimplicial objects proper, without losing the structure that forms their realizations into a symmetric spectrum. This result makes our model of the Dennis trace more robust, because it is insensitive to which point-set model of $\THH$ we use (see \cref{rmk:properness_for_zigzag}).

Let $\mc I$ be the category with one object $\underline{n} = \{1,\dotsc,n\}$ for every integer $n \geq 0$, where $\underline{0} = \emptyset$, and morphisms the injective maps of finite sets $\underline{m} \arr \underline{n}$ (which need not preserve order). For each $n \geq 0$ let $\mathbf\Delta^{\op \times n}$ be the $n$-fold product of the usual simplicial indexing category; that is, the opposite of the category of nonempty totally ordered finite sets
\[[k] = \{0 \to \dotsm \to k\}\]
 for $k \geq 0$ and order-preserving functions.

For each morphism $f\colon \underline{m} \arr \underline{n}$ of $\mc I$ there is a functor $f_*\colon \mathbf\Delta^{\times m} \arr \mathbf\Delta^{\times n}$ taking
\[([k_1],\dotsc,[k_m])\]
 to the $n$-tuple whose value at $f(i)$ is $[k_i]$ and whose value outside the image of $f$ is $[1]$.
In particular, when $m = n$ this gives an action of the symmetric group $\Sigma_n$ on $\mathbf\Delta^{\op \times n}$. This rule forms a strict diagram of categories indexed by $\mc I$. Therefore we may take its Grothendieck construction $\mc I \int \mathbf\Delta^{\op\times -}$, a category whose objects are tuples $(m;k_1,\dotsc,k_m)$ and whose morphisms consist of an injection $f\colon \underline{m} \arr \underline{n}$ in $\mc I$ and a morphism
\[(\phi_i) \colon f_*([k_1],\dotsc,[k_m]) \to ([l_1],\dotsc,[l_n])\] in $\mathbf\Delta^{\op\times n}$.

\begin{defn}\label{sigma-delta-diagram}
	Given a pointed category $\mc M$, a \textbf{$\Sigma_{\mathbf\Delta}$-diagram} in $\mc M$ is a functor $X_{(\bullet;\bullet,\dotsc,\bullet)}$ from the Grothendieck construction $\mc I \int \mathbf\Delta^{\op\times -}$ to $\mc M$ satisfying:
	\begin{itemize}
		\item $X_{(n;k_1,\dotsc,k_n)} \cong *$ any time $k_i = 0$ for at least one $i$, and
		\item the morphisms $(m;k_1,\dotsc,k_m) \arr (n;f_*(k_1,\dotsc,k_m))$ with every $\phi_i = \id$ induce isomorphisms
		\[ X_{(m;k_1,\dotsc,k_m)} \cong X_{(n;f_*(k_1,\dotsc,k_m))}. \]
	\end{itemize}
\end{defn}

Let $\mc M$ be a simplicially tensored pointed model category.  Under these assumptions, we may take the geometric realization of a simplicial object in $\mc M$, symmetric spectrum objects in $\mc M$ are well-defined, and multisimplicial objects in $\mc M$ have a Reedy model structure. The most important example in \S\ref{sec:dennis_trace} is when $\mc M$ is the category of simplicial objects in orthogonal spectra, the extra simplicial direction accommodating the $w_\bullet$ construction.

We define a functor from $\Sigma_{\mathbf\Delta}$-diagrams $X_{(\bullet;\bullet,\dotsc,\bullet)}$ to symmetric spectrum objects in $\mc M$ by taking level $n$ to be the geometric realization $|X_{(n;\bullet,\dotsc,\bullet)}|$. To give the action of the indexing category $\Sigma_S$ for symmetric spectra \cite[4.2]{mandell_may_shipley_schwede}, we observe that each injective map $f\colon \underline{m} \arr \underline{n}$ produces a map
\[ \Sigma^{\underline{n} - f(\underline{m})} |X_{(m;\bullet,\dotsc,\bullet)}| \arr |X_{(n;\bullet,\dotsc,\bullet)}| \]
by identifying the left-hand side with the realization of the sub-object of the right-hand side in which every index $l_i$ with $i \not\in f(\underline{m})$ is restricted to the values of 0 and 1, while the indices in $f(\underline{m})$ take all values. We simply call this operation $|-|$, since it amounts to taking realization and then making some observations. The proof of the following is pure bookkeeping.

\begin{lem}\label{realization_is_symmetric_spectrum}
	The above gives a well-defined functor $|-|$ from $\Sigma_{\mathbf\Delta}$-diagrams in $\cM$ to symmetric spectrum objects in $\cM$.
\end{lem}

Next we prove the properness theorem. We say that a morphism of $\Sigma_{\mathbf\Delta}$-diagrams is a {\bf level-wise equivalence} if it is a weak equivalence in $\cM$ on each term $X_{(n;k_1,\dotsc,k_n)}$.

\begin{thm}\label{properness-thm}
	The functor $|-|$ is left-deformable, meaning that there is a class of $\Sigma_{\mathbf\Delta}$-diagrams on which it preserves level-wise equivalences and that there is a cofibrant replacement functor with image in this class.
\end{thm}

\begin{proof}
	The category of simplicial objects in $\mc M$ admits a Reedy model structure in which
	\begin{enumerate}
		\item An object $X_{\bullet}$ is cofibrant when the latching maps
    \[
    L_n X \coloneqq \bigcup_{\substack{\text{degeneracies} \\ s \colon X_{m} \to X_{n}}} X_{m} \arr X_n
		\]
    are all cofibrations,
    \item\label{item:induct_cofib} cofibrant replacement can be performed inductively over the simplicial levels, and
		\item geometric realization preserves weak equivalences between cofibrant objects.
	\end{enumerate}
	The upshot of \cref{item:induct_cofib} is that on the subcategory of simplicial objects where the first two latching maps $* \to X_0$ and $L_1 X = X_0 \to X_1$ are already cofibrations, there is a cofibrant replacement functor that does not modify the first two levels.

	This generalizes to multisimplicial objects. There is a Reedy model structure on $n$-fold multisimplicial objects, obtained inductively as the Reedy model structure on simplicial $(n-1)$-fold multisimplicial objects, where the latching object at multilevel $(k_1,\dotsc,k_n)$ is a colimit over the multilevels $(l_1,\dotsc,l_n)$ with all $l_i \leq k_i$ and some $l_i \neq k_i$. In particular, if the latching maps are already cofibrations for all multilevels with at least one $k_i \leq 1$, then there is a cofibrant replacement that does not change those multilevels, only those with all $k_i \geq 2$.

	We extend this to a cofibrant replacement on $\Sigma_n \int \Delta^{\op\times n}$-diagrams in $\mc M$ by carrying out the factorization at one term $(k_1, \dotsc, k_n)$ in each $\Sigma_n$-orbit of the objects of $\Delta^{\op\times n}$, and extending to the remaining terms in the orbit by $\Sigma_n$-equivariance. To be precise, we pick a factorization
	\[ \xymatrix{
		L_{(k_1, \dotsc, k_n)} QX \ar@{>-->}[r] &
		QX_{(k_1, \dotsc, k_n)} \ar@{-->>}[r]^-\sim &
		X_{(k_1, \dotsc, k_n)} \times_{M_{(k_1, \dotsc, k_n)} X} M_{(k_1, \dotsc, k_n)} QX \ar@{->>}[r]^-\sim &
		X_{(k_1, \dotsc, k_n)}
	} \]
	in spaces that are equivariant with respect to the subgroup of $\Sigma_n$ that fixes $(k_1, \dotsc, k_n)$. Then for $(k'_1,\dotsc,k'_n) = \sigma(k_1, \dotsc, k_n)$ we define the cofibrant replacement to be $QX_{(k_1, \dotsc, k_n)}$ with structure maps
	\[ \xymatrix{
		L_{(k'_1,\dotsc,k'_n)} QX \ar[r]^-{\sigma^{-1}} & L_{(k_1, \dotsc, k_n)} QX \ar[r] & QX_{(k_1, \dotsc, k_n)} \ar[r] & X_{(k_1, \dotsc, k_n)} \ar[r]^-\sigma & X_{(k'_1,\dotsc,k'_n)}.
	} \]
	Note that the $\Sigma_n$-equivariance of the diagrams on the left and right make the composite into a factorization of the given map $L_{(k'_1,\dotsc,k'_n)} QX \to X_{(k'_1,\dotsc,k'_n)}$. It is then straightforward to define the action of the morphisms of $\Sigma_n \int \mathbf\Delta^{\op\times n}$ on these replacements so that they form a cofibrant diagram level-wise equivalent to $X$.

	We extend this further to a cofibrant replacement on $\mc I_{\leq n} \int \Delta^{\op\times -}$ by induction on $n$. At each stage we extend the replacement from $\mc I_{\leq (n-1)} \int \Delta^{\op\times -}$ to those objects of $\Delta^{\op\times n}$ where at least one index is $\leq 1$, by declaring that each order-preserving map $\alpha\colon \underline m \to \underline n$ will go to an identity morphism in the cofibrant replacement, and then carefully checking that there is a unique and well-defined way to extend this definition to the remaining morphisms. We then extend to the rest of $\mc I_{\leq n} \int \Delta^{\op\times -}$ using the inductive cofibrant replacement from the previous paragraph. Using the structure of the Grothendieck construction we check that everything is well-defined and commutes with the map back to $X$. By construction, on each of the fiber categories $\mathbf\Delta^{\op\times n}$ the replacements are Reedy cofibrant, so the realization will preserve equivalences. Finally, we check the construction is natural in maps $X \to Y$, so that the cofibrant replacement is a functor mapping naturally back to the identity functor. This finishes the proof.
\end{proof}

\section{The Dennis trace}\label{sec:dennis_trace}

We conclude by constructing the Dennis trace map
$K(\End(\uncat\cC)) \to \THH(\spcat\cC)$ for a spectrally enriched Waldhausen category $\spcat\cC$. This is generalized to handle coefficients and lifted to $\TR$ in \cite{clmpz-dt}, but we give a short presentation here that highlights the role of the technical material developed in this paper.

\begin{defn}\label{def:category_of_endomorphisms}
	Given a Waldhausen category $\uncat{\cC}$, let $\End(\uncat\cC)$ be the
	Waldhausen category of functors $\Fun(\bbN,\uncat\cC)$, where $\bbN$ is
	considered as a category with one object and morphism set $\bbN$.  More
	concretely, the objects of $\End(\uncat{\cC})$ are endomorphisms
	$f\colon a \arr a$ in $\uncat{\cC}$, and the morphisms are commuting squares of the
	form
	\[ \xymatrix @R=1.5em{
		a \ar[r]^-f \ar[d]_-i & a \ar[d]^-{i} \\
		b \ar[r]_-g & b. } \]
	We define such a morphism to be a cofibration or
	weak equivalence if $i$ is a cofibration or weak equivalence, respectively.
	We also define exact functors
	\[
	\begin{tikzcd}
	\uncat{\cC} \arrow[r, bend left=50, "\iota_0"{below}]
	\arrow[r, bend right=50, "\iota_1"]
	&\End(\uncat{\cC}) \arrow[l, ""]
	\end{tikzcd}
	\]
	where $\End(\uncat{\cC}) \arr \uncat{\cC}$  forgets the endomorphism $f$.  The inclusions $\iota_0,\iota_1\colon \uncat{\cC} \arr \End(\uncat{\cC})$ equip each object $a$ with either the zero endomorphism or the identity endomorphism.
\end{defn}

\begin{example}
	If $A$ is a ring spectrum and $\spcat\cC = \tensor[^A]{\Perf}{}$ is the
	spectral Waldhausen category of perfect $A$-modules from \cref{ex:Perf_as_spectral_wald_cat}, then $K(\uncat\cC)$ is the usual definition of algebraic $K$-theory $K(A)$, and the $K$-theory of $\End(\uncat{\cC})$ is the $K$-theory of endomorphisms $K(\End(A))$.
\end{example}

Let $\spcat\cC$ be a spectral Waldhausen category.
The key observation underlying the construction of the Dennis trace is that the inclusion of the $0$-simplices in the cyclic bar construction (see \cref{def:THH}) defines a canonical map
\begin{equation} \label{eq:k_end_trace_0}
\bigvee_{c_0\in \ob \uncat{\cC}} \cC(c_0,c_0) \to \THH(\cC).
\end{equation}
Each object $f \colon c_0 \to c_0$ of $\End(\uncat{\cC})$ defines a map of spectra $\bbS \arr \spcat{\cC}(c_0,c_0)$, and so composing with \eqref{eq:k_end_trace_0} gives a map
\begin{equation}\label{eq:k_end_trace_1}
\Sigma^{\infty} \ob \End(\uncat{\cC}) = \bigvee_{\substack{f\colon c_0 \to c_0,\\ c_0 \neq *}} \bbS \arr \bigvee_{c_0 \in \ob \uncat{\cC}} \spcat{\cC}(c_0,c_0) \arr \THH(\spcat{\cC})
\end{equation}
where $f$ runs over the objects of $\End(\uncat{\cC})$.
Applying \eqref{eq:k_end_trace_1} to the spectral Waldhausen category $w_{k_0}S^{(n)}_{k_1,\dotsc,k_n} \spcat{\cC}$
for each value of $n$ and $k_0, \dotsc, k_n$ defines a map of orthogonal spectra
\begin{equation}\label{eq:k_end_trace_2}
\Sigma^{\infty} \ob \End(w_{k_0}S^{(n)}_{k_1,\dotsc,k_n} \uncat{\cC}) \arr \THH(w_{k_0}S^{(n)}_{k_1,\dotsc,k_n} \spcat{\cC}).
\end{equation}

The target of the map in \eqref{eq:k_end_trace_2} is known to split into a wedge of copies of $\THH(\cC)$, by the additivity theorem for $\THH$ \cite{dundas_mccarthy,dundas_goodwillie_mccarthy,blumberg_mandell_published,blumberg_mandell_unpublished}. A streamlined proof of additivity using shadows in bicategories can be found in \cite[\dtref{Theorem 5.9}{\cref{dt-thm:THH_additivity}}]{clmpz-dt}.

\begin{thm}[The additivity theorem for $\THH$]\label{thm:THH_additivity}
  Given a spectral Waldhausen category  $\cC$,
  there is an equivalence of spectra
	\[
	\bigvee_{\substack{1 \leq i_j \leq k_j \\ 1 \leq j \leq n}} \THH(\spcat{\cC})
	\stackrel\sim\arr \THH(w_{k_0}S^{(n)}_{k_1, \dotsc, k_n} \spcat{\cC})
	\]
	induced by the inclusions of spectral Waldhausen categories
	\[ \iota_{i_{j}} \colon \spcat{\cC} \arr w_{k_0}S^{(n)}_{k_1, \dotsc, k_n} \spcat{\cC} \]
	that are constant in the $w_{k_0}$ direction and embed a copy of $\cC$ in the $i_j$-th column of the $j$-th $S_{\bullet}$ construction.
\end{thm}
Appending the splitting from \cref{thm:THH_additivity} to the map \eqref{eq:k_end_trace_2} gives a zig-zag of orthogonal spectra
\begin{equation}\label{eq:levelwise_trace_zigzag}
\Sigma^{\infty} \ob \End(w_{k_0}S^{(n)}_{k_1,\dotsc,k_n} \uncat{\cC}) \arr \THH(w_{k_0}S^{(n)}_{k_1, \dotsc, k_n} \spcat{\cC}) \overset{\simeq}{\longleftarrow} \bigvee_{\substack{1 \leq i_j \leq k_j \\ 1 \leq j \leq n}} \THH(\spcat{\cC}).
\end{equation}
The number of summands on the right is the same as the number of nonzero points
in the set $S^1_{k_1} \sma \dotsm \sma S^1_{k_n}$, where $S^1_\bullet$ is the
simplicial circle $\Delta[1]/\partial \Delta[1]$. Therefore these wedge sums
form an $(n + 1)$-fold multisimplicial spectrum $(S^1_\bullet)^{\sma n} \sma \THH(\spcat{\cC})$ that is constant in the $k_0$
direction.

The following lemma follows directly from the definitions and \cref{thm:THH_additivity}:

\begin{lem}\label{zigzag_respects_simplicial_str}
	The maps in the zig-zag \eqref{eq:levelwise_trace_zigzag} of multisimplicial orthogonal spectra commute with the $\Sigma_n$-actions and
	the identifications that remove a simplicial direction when its index is equal
	to 1.  In other words the given maps form a zig-zag of $\Sigma_\Delta$-diagrams of
	simplicial orthogonal spectra.
\end{lem}

Another way of saying this is that
$(S^1_\bullet)^{\sma n} \sma \THH(\spcat{\cC})$ is the free
$\Sigma_\Delta$-diagram on the spectrum $\THH(\spcat\cC)$ at level $(0;)$, and
the map inducing the splitting in the additivity theorem agrees with the map that arises from the
free-forgetful adjunction.

The geometric realization of each of these objects is a symmetric spectrum object in orthogonal spectra. We call such an object a {\bf bispectrum}. Bispectra are equivalent to orthogonal spectra, along a prolongation functor that turns every bispectrum into an orthogonal spectrum; see \cref{sec:model_structures} for details.

At level $n$ in the symmetric spectrum
direction, the resulting zig-zag of orthogonal spectra is
\begin{equation}\label{eq:levelwise_trace_after_realization}
\begin{aligned}
\abs{\Sigma^{\infty}
	\ob \End(w_{\bullet}S^{(n)}_{\bullet}
		\uncat{\cC})} &\arr
\abs{\THH(w_{\bullet}S^{(n)}_{\bullet} \spcat{\cC})} \overset{\simeq}{\longleftarrow} \Sigma^n \THH(\spcat{\cC}).
\end{aligned}
\end{equation}
There is a canonical identification of sets
\[
\ob \End(w_{\bullet}S^{(n)}_{\bullet}
	\uncat{\cC}) = \ob w_{\bullet}S^{(n)}_{\bullet}\End(\uncat{\cC})
\]
which identifies the bispectrum
on the left of \eqref{eq:levelwise_trace_after_realization} with the orthogonal
suspension spectrum of the symmetric spectrum $K(\End(\uncat{\cC}))$.  On the
other hand, the spectrum on the right of
\eqref{eq:levelwise_trace_after_realization} is the symmetric suspension
spectrum of the orthogonal spectrum $\THH(\spcat{\cC})$.

Applying the (left-derived) prolongation functor from \cref{quillen-adjoints}
to these bispectra, we get a zig-zag of orthogonal spectra
\begin{equation}\label{eq:trc_1_step_back}
\bbP K(\End(\uncat{\cC})) \arr \bbP\abs{\THH(w_{\bullet}S^{*}_{\bullet}
	\spcat{\cC})} \overset{\simeq}{\longleftarrow} \THH(\cC).
\end{equation}
The first term in \eqref{eq:trc_1_step_back} is the prolongation of $K(\End(\uncat{\cC}))$ from symmetric to orthogonal spectra.

\begin{defn} \label{twisted_dennis_trace} The \textbf{Dennis trace map}
	associated to a spectral Waldhausen category $\cC$ is obtained by choosing an
	inverse to the wrong-way map in \eqref{eq:trc_1_step_back}, defining a map
	\begin{equation}\label{eq:trc_ors_k_thh}
	\trc \colon \bbP K(\End(\uncat{\cC})) \arr \THH(\spcat{\cC})
	\end{equation}
	in the homotopy category of orthogonal spectra,
	or, equivalently, in the homotopy category of symmetric spectra
	\begin{equation}\label{eq:trc_ss_k_thh}
	\trc \colon K(\End(\uncat{\cC})) \arr \bbU\THH(\spcat{\cC})
	\end{equation}
	to the underlying symmetric spectrum of $\THH$.
\end{defn}

\begin{rmk}\label{rmk:properness_for_zigzag}
	In order to justify that the backwards map of the zig-zag is an equivalence after realization in \eqref{eq:levelwise_trace_after_realization}, we need to know that all three multisimplicial orthogonal spectra in the zig-zag are Reedy cofibrant, as discussed in \cref{sec:properness}.
	For the two outside terms this is straightforward to check. For the middle term it may not be true; however by \cref{properness-thm}
	we can always fatten the zig-zag to an equivalent one for which this holds. As a result of the properness theorem, the construction of the Dennis trace is  insensitive to the choice of model for $\THH$.
\end{rmk}

\begin{example}
	Taking $\spcat\cC = \tensor[^A]{\Perf}{}$ for a ring spectrum $A$, we get maps
	\[ \xymatrix@R=1em{
		K(A) \ar@{=}[d] & \widetilde K\End(A) \ar@{=}[d] & \ar[d]^-\sim \THH(A) \\
		K(\tensor[^A]{\Perf}{}) \ar[r]^-{\iota_1} & \widetilde K\End(\tensor[^A]{\Perf}{}) \ar[r]^-{\trc} & \THH(\tensor[^A]{\Perf}{})
	} \]
	whose composite agrees with the Dennis trace map $K(A) \to \THH(A)$ studied previously \cite{dundas_mccarthy,dundas_goodwillie_mccarthy,madsen_survey}.  The right vertical equivalence is an instance of the Morita invariance of $\THH$, as in \cref{thm:THH_morita_invariant}.
\end{example}

\appendix
\section{Model categories of bispectra}\label{sec:model_structures}
As we saw in \S\ref{sec:dennis_trace}, two spectral directions arise naturally in the construction of the Dennis trace map for a spectral Waldhausen category $\cC$: the orthogonal spectrum structure of $\THH(\cC)$ coming from the enrichment of $\cC$ and the symmetric spectrum structure coming from the iterated $S_{\bullet}$-construction. The Dennis trace is therefore naturally a map of bispectra.

This appendix collects several key results about bispectra and diagram spectra that are used in \S\ref{sec:dennis_trace} and \cite{clmpz-dt}, most importantly \cref{quillen-adjoints}, which details the equivalences between bispectra and orthogonal spectra. We give these results for $G$-equivariant bispectra that are naive in one direction and genuine in the other, since this structure is needed in the companion paper \cite{clmpz-dt}. Other sources that discuss bispectra include \cite{mandell_may_shipley_schwede,hovey:spec_sym_spec,dmpsw}.  A reader who has spent time with \cite{mandell_may_shipley_schwede,mandell_may} will probably find this section familiar.

We begin by recalling the indexing categories for diagram spectra from \cite{mandell_may_shipley_schwede,mandell_may}.
Let $\Sigma_S$ be the topological category with objects the natural numbers, regarded as the finite sets $\underline{n} = \{1,\dotsc,n\}$, and morphism spaces
\[ \Sigma_S(\underline{m},\underline{n}) = \bigvee_{i\colon \underline{m} \hookrightarrow \underline{n}} S^{\underline{n} - i(\underline{m})}. \]
Topological diagrams on $\Sigma_S$ encode the same data as symmetric spectra.

Let $\mathscr J$ be the topological category with objects the natural numbers, regarded as the Euclidean spaces $\bbR^n$, and morphism spaces
\[ \mathscr J(\bbR^m,\bbR^n) = O(\bbR^m,\bbR^n)_+ \sma_{O(n-m)} S^{n-m}. \]
Here $O(\bbR^m,\bbR^n)$ is the space of linear isometric embeddings $\bbR^m \rightarrow \bbR^n$, and the above space is the Thom space of the bundle that assigns an embedding to its orthogonal complement.
Topological diagrams on $\mathscr J$ encode the same data as orthogonal spectra.

Fix a finite group $G$, and let $\mathscr J_G$ be the category enriched in $G$-spaces whose objects are finite-dimensional $G$-representations $V$ in some fixed universe $\mathcal U$, morphisms as above for $\mathscr J$ with $G$ acting by conjugating the map $V \rightarrow W$ and acting on the vector in the complement. $G$-enriched diagrams on $\mathscr J_G$ encode the same data as orthogonal $G$-spectra.

Continuing to fix a finite group $G$, let $\Sigma_S \sma \mathscr J_G$ be the smash product category, whose objects are pairs $(\underline{n},V)$ and whose morphism spaces are the smash product
\[ \left( \bigvee_{i\colon \underline{m} \hookrightarrow \underline{n}} S^{\underline{n} - i(\underline{m})} \right) \sma \left( O(V,W)_+ \sma_{O(W-V)} S^{W-V} \right). \]
Topological diagrams on this category are called (symmetric-orthogonal) {\bf equivariant bispectra}. They are symmetric spectrum objects in orthogonal $G$-spectra. Let $G\Bi$ be the category of such bispectra. When $G$ is the trivial group, we simply call these objects bispectra.

It is a standard fact that orthogonal $G$-spectra are nothing more than $G$-objects in the category of orthogonal spectra. The following is the analog for bispectra.
\begin{lem}
	$G$-enriched diagrams on $\Sigma_S \sma \mathscr J_G$ are equivalent to $G$-objects in diagrams on $\Sigma_S \sma \mathscr J$. More informally, an equivariant bispectrum is the same thing as a bispectrum with a $G$-action.
\end{lem}

\begin{proof}
	Identical to the argument from  \cite[V.1.5]{mandell_may}. The forgetful functor that restricts from $\mathscr J_G$ to the subcategory $\mathscr J$ induces an equivalence on categories of $G$-enriched diagrams, essentially because when we forget the $G$-actions $\mathscr J \arr \mathscr J_G$ is an equivalence of topological categories.
\end{proof}

For every $G$-space $A$, integer $n \geq 0$ and finite-dimensional $G$-representation $V \subset \mathcal U$, we define the {\bf free bispectrum} $F_{(n,V)} A$ as the free diagram
\[ \Sigma_S(\underline{n},-) \sma \mathscr J_G(V,-) \sma A. \]
A map of equivariant bispectra $X \arr Y$ is a {\bf level-wise equivalence} or {\bf level-wise fibration} if the map $X(n,V)^H \arr Y(n,V)^H$ is an equivalence, respectively a Serre fibration, for all $n$ and $V$, and $H \leq G$.

\begin{prop}\label{level_model_structure}
	There is a cofibrantly generated level model structure on $G\Bi$ using the level-wise equivalences, the level-wise fibrations, and generating cofibrations and acyclic cofibrations
\begin{align*}
	I &= \left\{ \ F_{(n,V)}\left[ G/H \times S^{k-1} \rightarrow G/H \times D^k \right]_+\bigg| n,k \geq 0,  V \in \mathcal U,  H \leq G \right\} \\
	J &= \left\{ \ F_{(n,V)}\left[ G/H \times D^k \rightarrow G/H \times D^k \times I \right]_+ \bigg| n,k \geq 0,  V \in \mathcal U,  H \leq G \right\}.
\end{align*}
\end{prop}

We call the cofibrations in this model structure \textbf{free cofibrations}.

\begin{proof}
	Following \cite{hovey}, an {\bf $I$-cell complex} is a map built as a sequential colimit of pushouts of coproducts of maps in $I$. An {\bf $I$-injective map} is a map with the right lifting property with respect to $I$, and an {\bf $I$-cofibration} is a map with the left lifting property with respect to $I$-injective maps. It is then straightforward to check the following list of sufficient conditions:
	\begin{enumerate}
		\item The weak equivalences are closed under 2-out-of-3 and retracts.
		\item A map from a domain in $I$ into an $I$-cell complex factors through some finite stage, and the same condition for $J$.
		\item A $J$-cell complex is both a level-wise equivalence and an $I$-cofibration.
		\item A map is $I$-injective if and only if  it is both a level-wise equivalence and $J$-injective.
	\end{enumerate}
\end{proof}

A {\bf $\pi_*$-isomorphism} of equivariant bispectra is a map $X \arr Y$ inducing isomorphisms on the equivariant stable homotopy groups
\[ \pi_k^H(X) = \underset{(n,V)}\colim\, \pi_0([\Omega^{n+k+V} X(n,V)]^H) \]
for all $k \in \bbZ$ and $H \leq G$. The colimit is taken along the filtered system with one morphism $(m,V) \arr (n,W)$ when $m \leq n$ and $V \subseteq W \subset \mathcal U$, and no morphisms otherwise. The maps all arise from the action of the category $\Sigma_S \sma \mathscr J_G$, but restricted to the standard inclusions $\underline{m} \hookrightarrow \underline{n}$ hitting only the first $m$ elements, and the inclusions $V \subseteq W$ as subspaces of $\mathcal U$. When $k < 0$ we further restrict to $n \geq |k|$ and interpret $\Omega^{n+k+V}$ as $\Omega^{n+k}\Omega^V$, or restrict to $V$ containing a standard copy of $\bbR^{|k|}$ and interpret $\Omega^{n+k+V}$ as $\Omega^n\Omega^{V-\bbR^{|k|}}$. These two interpretations give naturally isomorphic groups.

An equivariant bispectrum $X$ is an \textbf{$\Omega$-spectrum} if the canonical maps $$X(n,V) \arr \Omega^{m+W}X(m+n,V+W)$$ are all equivariant weak equivalences. A \textbf{stable equivalence} of equivariant bispectra is a map $X \to Y$ inducing an isomorphism $[Y,Z] \to [X,Z]$ for all $\Omega$-spectra $Z$, where $[-,-]$ denotes maps in the level homotopy category.

We construct a model structure with the stable equivalences just as in \cite{mandell_may}. To streamline the exposition we present the minimal list of preliminary results needed to get the model structure, in an order that matches both the way they are proven and the way they are used in the proof of the model structure. Throughout, we define homotopy cofibers $Cf = (I \sma X) \cup_X Y$, homotopy fibers $Ff = X \times_Y F(I,Y)$, smash products $A \sma -$ and homs $F(A,-)$ for $G$-spaces $A$, by applying the usual constructions at each bispectrum level.
\begin{prop}\label{prop:long_list}\hfill
	\begin{enumerate}
		\item\label{item:long_coprod} A coproduct of stable equivalences is a stable equivalence.
		\item \label{item:long_po} A pushout of a stable equivalence is a stable equivalence, provided one of the two legs is a free cofibration.
		\item\label{item:long_count}  A countable composition of maps that are stable equivalences and free cofibrations is again a stable equivalence.
		\item\label{item:long_stable}  If $f$ is a stable equivalence then $Cf$ is stably equivalent to zero.
		\item\label{item:long_unit}  There are natural isomorphisms $\pi_{k+1}^H(\Sigma X) \cong \pi_k^H(X) \cong \pi_{k-1}^H(\Omega X)$. Composing these together, the unit $X \to \Omega\Sigma X$ and counit $\Sigma\Omega X \to X$ are both $\pi_*$-isomorphisms.
		\item\label{item:long_les}  There is a functorial way to assign to each map of bispectra $f\colon X \arr Y$ two long exact sequences
		\[ \xymatrix @R=.3em{
			\dotsc \ar[r] & \pi_k^H(Ff) \ar[r] & \pi_k^H(X) \ar[r] & \pi_k^H(Y) \ar[r] & \pi_{k-1}^H(Ff) \ar[r] & \dotsc \\
			\dotsc \ar[r] & \pi_{k+1}^H(Cf) \ar[r] & \pi_k^H(X) \ar[r] & \pi_k^H(Y) \ar[r] & \pi_k^H(Cf) \ar[r] & \dotsc
		} \]
		\item\label{item:long_pi_star}  There is a natural $\pi_*$-isomorphism $Ff \overset\sim\arr \Omega Cf$.
		\item\label{item:long_omega}  Every level-wise equivalence or $\pi_*$-isomorphism is a stable equivalence.
		\item\label{item:long_level}  For $\Omega$-spectra, every stable equivalence is a level-wise equivalence.
	\end{enumerate}
\end{prop}
\begin{proof}
	\cref{item:long_coprod,item:long_po,item:long_count} follow because $[-,Z]$ turns coproducts to products, pushouts to pullbacks, and sequential colimits to $\lim^1$ exact sequences.  The given assumptions ensure that if we start with cofibrant spectra then the input to $[-,Z]$ will also always be cofibrant.
\cref{item:long_stable,item:long_level} also follow formally from the definition of stable equivalence.
\cref{item:long_unit,item:long_les,item:long_pi_star} are proven just as they are classically.
For \cref{item:long_omega}, as in \cite[8.8]{mandell_may_shipley_schwede} it suffices to form a right deformation $X \overset{q}\arr QX$ such that $q$ is a level-wise equivalence when $X$ is an $\Omega$-spectrum, and $Q$ sends $\pi_*$-isomorphisms to level-wise equivalences. We construct $Q$ here by doing the construction from \cite[8.8]{mandell_may_shipley_schwede} in the symmetric spectrum direction and observing that it is continuously natural, therefore gives a bispectrum, and then repeating the same construction in the orthogonal spectrum direction, in both steps identifying the level homotopy groups of the result with the colimit of the homotopy groups of the original in one direction or the other, to confirm that it sends $\pi_*$-isomorphisms to level-wise equivalences. It also preserves $\Omega$-spectra because $\Omega$ commutes with mapping telescopes up to equivalence.
\end{proof}

For arbitrary pairs $(i,V)$ and $(j,W)$, let $k_{(i,V),(j,W)}$ refer to the map of bispectra that is the inclusion of the front end of the mapping cylinder for the canonical map
\[ \lambda_{(i,V),(j,W)}\colon F_{(i+j,V+W)} S^{j+W} \arr F_{(i,V)} S^0. \]
\begin{lem}\label{lem:lambda}
		$\lambda_{(i,V),(j,W)}$ is a stable equivalence, as is $A \sma \lambda_{(i,V),(j,W)}$ for any based $G$-CW complex $A$.
\end{lem}

\begin{proof}
	Taking maps into $Z$ gives $F(A,Z)_{(i,V)} \arr \Omega^{j+W} F(A,Z)_{(i+j,V+W)}$, which is a stable equivalence for any $\Omega$-spectrum $Z$.
\end{proof}

\begin{prop}\label{stable_model_structure}
	There is a cofibrantly generated stable model structure on $G\Bi$ with the stable equivalences, and generating cofibrations and acyclic cofibrations
\begin{align*}
	I = &\left\{ \ F_{(n,V)}\left[ G/H \times S^{k-1} \rightarrow G/H \times D^k \right]_+ \bigg|  n,k \geq 0,  V \in \mathcal U,  H \leq G  \right\} \\
	J = &\left\{ \ F_{(n,V)}\left[ G/H \times D^k \rightarrow G/H \times D^k \times I \right]_+ \bigg| n,k \geq 0,  V \in \mathcal U,  H \leq G \ \right\} \\
	& \cup  \left\{ {\ k_{(i,V),(j,W)} \ \square \ \left[ G/H \times S^{k-1} \rightarrow G/H \times D^k \right]_+ \bigg| i,j,k \geq 0,  V,W \in \mathcal U,  H \leq G } \right\}.
\end{align*}
	A map $X \arr Y$ is a fibration in this model structure if and only if  it is a level-wise fibration and each square of the following form is a homotopy pullback.
	\begin{equation}\label{eq:stably_fibrant}
	\xymatrix @R=1.5em{
		X_{(m,V)}^H \ar[d] \ar[r] & (\Omega^{n+W} X_{(m+n,V+W)})^H \ar[d] \\
		Y_{(m,V)}^H \ar[r] & (\Omega^{n+W} Y_{(m+n,V+W)})^H
	}
	\end{equation}
\end{prop}
Here $\square$ refers to the pushout-product, constructed from the operation that smashes a bispectrum with a space to produce another bispectrum.

\begin{proof}
	As in the proof of \cref{level_model_structure} it suffices to check the following four points.
	\begin{enumerate}
		\item The weak equivalences are closed under 2-out-of-3 and retracts, because they are defined from a notion of homotopy group.
		\item A map from a domain in $I$ into an $I$-cell complex factors through some finite stage, and similarly for $J$. This is straightforward since the domains are finite unions of free spectra on compact spaces.
		\item A $J$-cell complex is both a weak equivalence and an $I$-cofibration. For the $I$-cofibration part we use standard properties of pushout-products to show the new maps in $J$ are also $I$-cell complexes. Using the facts about coproducts, pushouts, and compositions of stable equivalences from \cref{prop:long_list}, the other part boils down to the fact that the smash product of $k_{(i,V),(j,W)}$ and a finite $G$-CW complex such as $(G/H \times S^{k-1})_+$ is a stable equivalence, which is \cref{lem:lambda}.
		\item A map is $I$-injective if and only if it is both a stable equivalence and $J$-injective. We already know that $I$-injective is equivalent to being a levelwise acyclic Serre fibration. $J$-injective rearranges to every level being a Serre fibration, and in addition the pullback-hom from $k_{(m,V),(n,W)}$ is a weak equivalence of based spaces. This latter map is then equivalent to
		\[ X(m,V) \arr \Omega^{n+W} X(m+n,V+W) \times_{\Omega^{n+W} Y(m+n,V+W)} Y(m,V). \]
		So it is an equivalence precisely when the square \eqref{eq:stably_fibrant} is a homotopy pullback, assuming at all times that $X \arr Y$ is a level-wise fibration, see also \cite[4.8]{mandell_may}.

		So if $f\colon X \arr Y$ is $I$-injective, then it is a level-wise equivalence, implying both that it is a stable equivalence and that \eqref{eq:stably_fibrant} is a homotopy pullback. Since it is also a level-wise fibration, this implies it is $J$-injective. Going the other way, if $f$ is $J$-injective and a stable equivalence, then the squares \eqref{eq:stably_fibrant} are homotopy pullbacks, so the homotopy fiber $Ff$ is an $\Omega$-spectrum. Since $f$ is a stable equivalence, combining several of the points from \cref{prop:long_list}, we conclude that $Ff$ is level equivalent to zero, hence $f$ is a level-wise equivalence. We also know it is a level-wise fibration, and these together imply $f$ is $I$-injective.
	\end{enumerate}
\end{proof}

Next we relate equivariant bispectra back to simpler kinds of spectra. Recall that a \textbf{naive symmetric $G$-spectrum} is a symmetric spectrum with $G$-action, where an equivalence of such is a map $X \to Y$ for which the induced map of fixed point spectra $X^H \to Y^H$ is an equivalence for every $H \leq G$. These have a model structure cofibrantly generated by the maps
\begin{align*}
I =& \left\{ \ F_{n}\left[ G/H \times S^{k-1} \rightarrow G/H \times D^k \right]_+  \bigg| n,k \geq 0,  H \leq G \ \right\} \\
J =& \left\{ \ F_{n}\left[ G/H \times D^k \rightarrow G/H \times D^k \times I \right]_+ \bigg| n,k \geq 0,  H \leq G  \right\} \\
& \cup  \left\{ {\ k_{i,j} \ \square \ \left[ G/H \times S^{k-1} \rightarrow G/H \times D^k \right]_+ \bigg| i,j,k \geq 0,  H \leq G } \right\}.
\end{align*}
By contrast, a \textbf{genuine orthogonal $G$-spectrum} is an orthogonal $G$-spectrum in the model structure from \cite{mandell_may}, where a map is an equivalence if it induces isomorphisms on the homotopy groups
\[ \pi_k^H(X) = \begin{cases}\quad
\underset{V \subset \mathcal U}\colim\, \pi_k\left(\left[\Omega^V X(V)\right]^H\right)  &\text{if } k \geq 0 \\
\underset{V \subset \mathcal U,\ \bbR^{|k|} \subset V}\colim\, \pi_0\left(\left[\Omega^{V-\bbR^{|k|}} X(V)\right]^H\right)  &\text{if } k < 0.\end{cases} \]
Now consider the following three adjunctions.
\begin{enumerate}
	\item For each naive symmetric spectrum $X$ we create an equivariant bispectrum $\Sigma^\infty X$ by defining
	\[ (\Sigma^\infty X)_{n,V} = \Sigma^V X_n. \]
	The right adjoint of the functor $\Sigma^{\infty}$ restricts a bispectrum to the levels $(n,0)$ for all $n$ to produce a symmetric spectrum.
	\item From each genuine orthogonal $G$-spectrum $X$, we create an equivariant bispectrum $\Sigma^\infty X$ by the rule
	\[ (\Sigma^\infty X)_{n,V} = \Sigma^n X_V. \]
	The right adjoint of the functor $\Sigma^{\infty}$ restricts a bispectrum to the levels $(0,V)$ to produce an orthogonal spectrum.
	\item Choose an isomorphism $\bbR^\infty \oplus \mathcal U \cong \mathcal U$, and define a functor $\Sigma_S \sma \mathscr J_G \arr \mathscr J_G$ that sends $(n,V)$ to the image of $\bbR^n \oplus V$ under this isomorphism. The restriction and left Kan extension along this functor give an adjunction between equivariant bispectra and orthogonal $G$-spectra. We call the left adjoint \textbf{prolongation} $\bbP(-)$ and the right adjoint the \textbf{shift bispectrum} $sh(-)$, following \cite{dmpsw}. Concretely, $\bbP X$ is the coequalizer
\begin{small}
	\[ \bigvee_{(m,V),(n,W)} \mathscr J_G(\bbR^n \oplus W,-) \sma \Sigma_S(m,n) \sma \mathscr J_G(V,W) \sma X_{(m,V)}
	\rightrightarrows \bigvee_{(m,V)} \mathscr J_G(\bbR^m \oplus V,-) \sma X_{(m,V)}
	\rightarrow \bbP X \]
\end{small}
	and $sh(Y)_{(n,V)} = Y_{\bbR^n \oplus V}$. Up to isomorphism, these functors do not depend on the original choice of isomorphism $\bbR^\infty \oplus \mathcal U \cong \mathcal U$.
\end{enumerate}

\begin{prop}\label{quillen-adjoints}
	The above three adjunctions are Quillen adjunctions, and the last two are Quillen equivalences.
\end{prop}

\begin{proof}
	It is straightforward to check that all three left adjoints preserve generating cofibrations and acyclic cofibrations. To verify that the last two are Quillen equivalences, we observe that they compose to give the identity functor on orthogonal $G$-spectra, so it suffices to show that the first one ($\Sigma^\infty$) gives an equivalence of homotopy categories. It has $\bbP$ as its left inverse, so it is faithful. To show it is full and essentially surjective, it suffices to construct a right inverse. To do this we take an $\Omega$-bispectrum $X$ and check that the natural map
	\begin{equation}\label{eq:equivalence_proof}
		\Sigma^\infty \textup{res} X \to X
	\end{equation}
	from the suspension spectrum of its restriction is a $\pi_*$-isomorphism. To prove this it suffices to show that the map of colimit systems defining $\pi_k^H$ is bijective on the colimits when restricted to any one symmetric spectrum level $n$, where it arises from a map of orthogonal spectra of the form
	\[ \Sigma^n X(0,-) \arr X(n,-). \]
	Since $X$ is an $\Omega$-spectrum, we can identify this with the counit map
	\[ \Sigma^n\Omega^n X(n,-) \arr X(n,-) \]
	which is always a $\pi_*$-isomorphism of orthogonal spectra as in \cref{prop:long_list}(\ref{item:long_unit}).
\end{proof}

Let $H \leq G$ be any subgroup. For simplicity we assume that $G$ is abelian so that its Weyl group is just $G/H$. Let $\mathscr J_G^H$ denote the category obtained from $\mathscr J_G$ by taking $H$-fixed points of all the mapping spaces. Let
\[ L\colon \mathscr J_G^H \arr \mathscr J_{G/H} \]
be the map of topological categories taking $V$ to $V^H$, acting on morphisms by
\[ i\colon V \arr W, \ w \in (W-i(V))^H \quad \leadsto \quad i^H\colon V^H \arr W^H, \ w \in W^H - i(V^H). \]
For each $G$-equivariant bispectrum $X$ the spaces $X(n,V)^H$ form a diagram on $\mathscr J_G^H$, that we call $\textup{Fix}^H X$.
\begin{defn}\label{gfp_bispectra}
	The geometric fixed point bispectrum $\Phi^H X$ is the left Kan extension of $\textup{Fix}^H X$ along the functor
\[ \xymatrix{ \Sigma_S \sma \mathscr J_G^H \ar[r]^-{\id \sma L} & \Sigma_S \sma \mathscr J_{G/H}. } \]
Equivalently, it is the functor $\Phi^H$ on orthogonal spectra, applied to the orthogonal spectrum $X_n$ at every symmetric spectrum level $n \geq 0$, which then assemble together into a bispectrum.
\end{defn}
The following has the same proof as in \cite{mandell_may}.
\begin{lem}\label{lem:gfp_properties}
	The functor $\Phi^H(-)$ is not a left adjoint, but preserves coproducts, sends pushouts along levelwise closed inclusions to pushouts, and sends sequential colimits along levelwise closed inclusions to colimits. Furthermore there are natural isomorphisms of $G/H$-spectra
	\[ \Phi^H F_{(n,V)} A \cong F_{(n,V^H)} A^H. \]
\end{lem}
\begin{cor}
	$\Phi^H$ preserves cofibrations and acyclic cofibrations.
\end{cor}
Therefore $\Phi^H$ preserves weak equivalences between cofibrant objects, so it has a well-defined left-derived functor by applying it to cofibrant inputs.
\begin{prop}\label{gfp_commutes}
	For naive symmetric $G$-spectra $X$ there is a natural isomorphism of $G/H$-bispectra
	\[ \Phi^H \Sigma^\infty X \cong \Sigma^\infty X^H. \]
	For orthogonal $G$-spectra $X$ there is a natural isomorphism of $G/H$-bispectra
	\[ \Phi^H \Sigma^\infty X \cong \Sigma^\infty \Phi^H X. \]
	For equivariant bispectra $X$ there is a natural map of orthogonal $G/H$-spectra
	\[ \bbP \Phi^H X \arr \Phi^H \bbP X \]
	that is an isomorphism when $X$ is cofibrant.
\end{prop}
\begin{proof}
	The first point follows just as in the isomorphism from \cref{lem:gfp_properties}, and the second point follows easily since geometric fixed points are being applied at each orthogonal spectrum level separately. For the last point we observe there is a canonical map of $\mathscr J_G^H$-diagrams
	\[ \bbP^H \textup{Fix}^H X \arr \textup{Fix}^H \bbP X \]
	that is an isomorphism when $X$ is cofibrant, the first term $\bbP^H$ denoting left Kan extension along the $H$-fixed points of the direct sum map $\Sigma_S \sma \mathscr J_G^H \arr \mathscr J_G^H$. Then we apply left Kan extension along $L$ to both sides, and commute the left Kan extensions on the left-hand side to get the desired natural map.
\end{proof}

The argument from \cite[\S 3]{malkiewich_thh_dx} also establishes the following rigidity lemma. Together with rigidity for orthogonal spectra, this proves that the above isomorphisms are unique and therefore we encounter no coherence issues when commuting them past each other.
\begin{lem}\label{lem:rigidity}
	Any functor from cofibrant $G$-bispectra to orthogonal spectra that is isomorphic to $\Phi^G \bbP X$ is uniquely isomorphic to $\Phi^G \bbP X$. The same applies to $\Phi^G X$ as a functor from cofibrant $G$-bispectra to bispectra.
      \end{lem}

\bibliography{references}
\bibliographystyle{amsalpha2}

\end{document}